\documentclass[11pt,reqno]{amsproc}
\usepackage{amsmath}
\usepackage{amssymb}
\usepackage{amsfonts}
\usepackage[margin=1in]{geometry}
\usepackage{amsthm}
\usepackage{mathrsfs}
\usepackage{enumitem}
\numberwithin{equation}{section}
\usepackage{xcolor}
\usepackage{hyperref}
\hypersetup{
    colorlinks=true,
    citecolor=magenta,
    linkcolor=blue,
    filecolor=teal,      
    urlcolor=cyan,
    pdfpagemode=FullScreen,
    }
 \newtheorem{theorem}{Theorem}[section]
 \newtheorem{proposition}[theorem]{Proposition}

 \newtheorem{definition}[theorem]{Definition}
 
 \theoremstyle{remark}
 \newtheorem{remark}[theorem]{Remark}

 \newcommand{\pa}{\partial}
\newcommand{\la}{\label}
\newcommand{\fr}{\frac}
\newcommand{\na}{\nabla}
\newcommand{\be}{\begin{equation}}
\newcommand{\ee}{\end{equation}}
\newcommand{\ba}{\begin{array}{l}}
\newcommand{\ea}{\end{array}}
\newcommand{\Rr}{{\mathbb R}}

\begin{document}

\title[On self-similarity for incompressible 3D Euler]{On putative self-similarity for incompressible 3D Euler}

\author{Peter Constantin} 
\address{Department of Mathematics, Princeton University, Princeton, NJ 08540}
\email{\href{https://web.math.princeton.edu/~const/}{const@math.princeton.edu}}

\author{Mihaela Ignatova}
\address{Department of Mathematics, Temple University, Philadelphia, PA 19122}
\email{\href{https://sites.temple.edu/ignatova/}{ignatova@temple.edu}}

\author{Vlad Vicol}
\address{Courant Institute of Mathematical Sciences, New York University, New York, NY, 10012}
\email{\href{https://cims.nyu.edu/~vicol/}{vicol@cims.nyu.edu}}

\begin{abstract}
We consider hypothetical solutions of 3D Euler which blow up in finite time  in a self-similar fashion. We prove that if the initial data has finite kinetic energy, then the similarity exponent $\gamma$ which governs the rate of zooming in must be at least $2/5$. If a smooth globally self-similar blowup profile exists, and this profile satisfies an outgoing property, we prove that  $\gamma \geq 1/2$. For axisymmetric solutions, we establish the bound $\gamma\geq 1/2$ under the sole assumption that the velocity profile is $C^2$ smooth.
\end{abstract}

\maketitle
 
\section{Introduction}

Self-similarity~\cite{Barenblatt96} is a powerful idea: it says that things look the same at different scales, if you just know how to zoom in or out. It is a particularly useful tool for identifying singular solutions of nonlinear PDE via a reduction to ODE~\cite{Sedov18,EggersFontelos15}. 
Self-similarity is found in compressible fluids   and in many other physical systems when dynamics are determined by local interactions. Compressible Euler equations for instance exhibit self-similar explosions~\cite{Sedov46,vonNeumann47,Taylor50}, implosions~\cite{Guderley42,Hunter60,MRRS22a,BCG25,SWWZ25} and shocks~\cite{BSV22,BSV23}.

Incompressible fluids are not local, the outside matters. This makes the link between possible singularities and local self-similar behavior rather tenuous.   There have been a number of recent works where self-similar incompressible singularities have been reported, either as rigorous proofs or in computational studies. In all of them there is a remnant of compression due to either the presence of boundaries~\cite{ChenHou22,ChenHou25a,ChenHou25} and~\cite{GuoLuo14,WLGZB23,WangEtAl25,WLLB25,WLLAH25}, or the lack of smoothness of vorticity~\cite{Elgindi21,ChenHou21,EGM22,ElgindiPasqua23,Chen24,CMZ25}. The goal of this paper is to analyze constraints on putative self-similar singularities for the incompressible 3D Euler equations with smooth initial data, in the absence of boundaries.

We recall the vorticity formulation of incompressible 3D Euler equations:
\begin{align}
\partial_t \omega + u \cdot \nabla \omega = \omega \cdot \nabla u, 
\qquad
\nabla \cdot u = 0  , 
\qquad 
\omega = \nabla \times u.
\label{eq:Euler}
\end{align}
The smooth and localized initial data $\omega_0$ for the Cauchy problem associated to~\eqref{eq:Euler} is specified at time $t=0$, and the equations are posed on the whole space ($x \in \mathbb{R}^3$). 
It is well known that singularities of any kind cannot arise in finite time from smooth and localized initial data unless the vorticity becomes infinite, such that its maximum magnitude is not time integrable. This is the well-known Beale-Kato-Majda criterion~\cite{BKM84}. We say that~\eqref{eq:Euler} {exhibits a self-similar singularity} if there exists a first time $T_* >0$, a space location $x_* \in \mathbb{R}^3$, a similarity exponent~$\gamma>0$, and a vorticity profile~$\Omega$ such that 
\begin{equation}
\label{eq:self-similar:intro}
\omega(x,t) = \frac{1}{T_* - t} \, \Omega \left( \frac{x-x_*}{(T_*-t)^\gamma} \right) + \mathrm{l.o.t.}
\qquad 
\mbox{as} 
\qquad t\to T_*^-.
\end{equation}
The factor $(T_*-t)^{-1}$ appearing in~\eqref{eq:self-similar:intro} is unavoidable, and consistent with~\cite{BKM84}.
A more precise meaning of a self-similar singularity for~\eqref{eq:Euler} will be given later in the paper; see Section~\ref{sec:two:fifths} for an asymptotically self-similar blowup, and Section~\ref{sec:self-similar:Euler} for a globally self-similar blowup. 

If a true self-similar singularity forms in~\eqref{eq:Euler}, the value of the exponent $\gamma$ is an important and powerful dynamic property, and needs to be understood. Finding a profile $\Omega$ is a well known challenge. Because access to self-similar solutions is fraught with numerical and theoretical difficulties, it is useful to establish strict mathematical guardrails and to provide computational benchmarks.

The Euler equations have a great variety of different types of solutions. Blow up of smooth solutions with infinite kinetic energy has been proved~\cite{child,Cric,gib,jtstuart}. In this paper we prove that if the initial data of~\eqref{eq:Euler} has finite kinetic energy, then a solution behaving as in~\eqref{eq:self-similar:intro} must have $\gamma \geq 2/5$; see Theorem~\ref{thm:two:fifths}.

Beyond a general intrinsic interest in lower bounds for $\gamma$ as a benchmark for ongoing computational and analytical studies, in the second half of the paper we focus on the hypothetical case  when $\gamma$ lies below the parabolic threshold of $1/2$, relevant for incompressible 3D Navier-Stokes. At $\gamma = 1/2$, the advection operator $\partial_t + u \cdot \nabla$ and the dissipative operator $-\Delta$ are in exact balance when acting on vorticities in the form~\eqref{eq:self-similar:intro}.  Liouville theorems ~\cite{NRS96,Tsai98,Seregin05,ChaeWolf17} rule out certain globally self-similar solutions  for the 3D incompressible Navier-Stokes equations with $\gamma=1/2$. Under reasonable assumptions, $\gamma >1/2$ cannot be the exponent of an approximate self-similar blow up for 3D incompressible Navier-Stokes equations (see Proposition~\ref{nse}). 
For self-similar scaling with  $\gamma < 1/2$, the viscous term yields a vanishingly small force competing with  the Euler nonlinearity.  If self-similar solutions of Euler equations with $\gamma< 1/2$ are found,
they can be used to obtain  3D Navier-Stokes blow up as a perturbation of the 3D Euler one,  allowing the inviscid blowup to be used to find a viscous one. This is indeed a rigorously proven fact in the realm of compressible flows: globally self-similar implosion singularities for 3D compressible isentropic Euler equations can be used to provide asymptotically self-similar implosion singularities for 3D compressible barotropic Navier-Stokes equations with constant viscosity coefficients; see~\cite{MRRS22a,MRRS22b} and~\cite{BCG25,SWWZ25}. In the process of establishing the blow up ~\cite{MRRS22b,BCG25,SWWZ25},
in addition to having $\gamma<1/2$ the authors prove and use that the inviscid similarity profile has infinite regularity (real-analyticity) in a neighborhood of all  sonic points (stagnation points on the self-similar fast-acoustic characteristics). 

An analogous approach to establish blow up for incompressible viscous flows hinges on finding self-similar blowup of~\eqref{eq:Euler} with $\gamma < 1/2$. This is advocated in~\cite{WLGZB23}  and~\cite{WangEtAl25}. 

The paper~\cite{Elgindi25} identifies an outgoing property (cf.~\eqref{eq:outgoing:tarek}) as being an essential ingredient for the existence of self-similar profiles. This outgoing property quantifies the statement that all self-similar Lagrangian trajectories that originate from nonzero labels must escape to infinity as (self-similar) time diverges. In particular, this condition implies that the self-similar Lagrangian flow has no stagnation points, except for the trivial one at the origin.   

In this paper we prove that if a globally self-similar smooth solution to incompressible 3D Euler equations exists, and it satisfies a local version of the outgoing property, then we must have $\gamma \geq 1/2$; see Theorem~\ref{thm:outgoing:one:half} and Theorem~\ref{thm:outgoing:global:new}. Thus, in this case, the approach proposed to produce incompressible Navier-Stokes singularities via Euler self-similar solutions must fail.

Independently of the outgoing property, if a globally self-similar vorticity profile $\Omega$ exists, then it cannot be too small, as measured by a certain scaling invariant norm, see Theorem~\ref{thm:Omega:not:small}. 

When restricting our analysis to axisymmetric similarity profiles, we prove a much stronger result. If the velocity field is $C^2$ smooth, then $\gamma \geq 1/2$; see Theorem~\ref{thm:axisym:final}. Additionally, if the velocity has nonzero swirl at a stagnation point of the self-similar Lagrangian flow map, then $\gamma=1/2$; see Theorem~\ref{thm:axissym:with:swirl}.

Our proofs are based on basic, well known properties of incompressible Euler and Navier-Stokes equations. We chose to present rigorous results, with explicit minimal assumptions. The fact that $\gamma\ge 2/5$ for finite kinetic energy flows, is such a result. The relevance of the exponent $\gamma= 1/2$ in Euler flows is  one of the main messages of this paper. This exponent, which is natural for Navier-Stokes self-similarity, is natural for Euler equations as well. This is due to the invariance of circulation, a quantity with dimension of kinematic viscosity.

 
\section{On asymptotically self-similar blowup with finite kinetic energy}
\label{sec:two:fifths}

A 3D incompressible Euler singularity cannot be just a vorticity magnitude singularity---the spatial gradient of vorticity must also blow up fast enough. For $\mu \in (0,1)$, let 
\begin{equation*}
\ell_\mu (t) := \min\left \{L_0; \left(\frac{[\omega (\cdot,t)]_{\mu}}{\|u(\cdot, t)\|_{L^2}}\right)^{-\frac{2}{2\mu +5}}\right\}
\end{equation*}
be the length scale formed with the vorticity H\"{o}lder seminorm $[\omega (\cdot, t)]_{\mu} := \sup_{0 < |x-y| \le L_0}\frac{|\omega (x,t)-\omega(y,t)|}{|x-y|^{\mu}}$ and the $L^2$ norm of velocity, where $L_0$ is an arbitrary reference length scale. It was shown in~\cite[Theorem 1, p. 38]{Constantin94a} that if
\begin{equation*}
\int_0^{T_*} \ell_\mu(t)^{-\frac{5}{2}}dt <\infty,
\end{equation*}
then no singularities can occur from smooth and localized data at time $T_*$. This result follows easily from the representation of vorticity 
 and from the conservation of kinetic energy. 
 
 In \cite{Constantin94a} it is also remarked as a consequence that if the vorticity is assumed to have a globally self-similar blowup, namely 
\begin{equation*}
\omega(x,t) = \frac{1}{T_*-t}\Omega\left(\frac{x}{L(t)}\right),
\end{equation*}
then we must have
\begin{equation*}
\int_0^{T_*} L(t)^{-\frac{5}{2}}dt = \infty.
\end{equation*}
In particular, if $L(t) =L_0(1 -\frac{t}{T_*})^\gamma$ as in~\eqref{eq:self-similar:intro}, then $\gamma\geq 2/5$ is necessary for blow up.  Here we revisit this result, without assuming the existence of a globally self-similar profile; instead, we assume only local behavior consistent with self-similarity. Furthermore, we link the putative exponent $\gamma$ to the behavior of the velocity away from the local self-similar behavior. In \cite{Chae07} it was shown that self-similar blow up cannot occur if the vorticity profile decays at infinity fast enough.\footnote{From a complementary viewpoint, for locally self-similar profiles in the corresponding range of exponents, it was shown in~\cite{BronziShvydkoy15} (see also~\cite{Shvydkoy13}) that a nontrivial finite-energy blowup must carry a positive amount of energy, in the sense that the local energy average admits a lower bound matching the upper bound dictated by finite kinetic energy.} This is not what we do here, our assumed behavior does not invoke a profile, but if a profile existed, it would have required only that the profile's first derivatives be bounded in a fixed ball. We  make assumptions about the nature of the true solution of~\eqref{eq:Euler}, which would be automatically satisfied if a self-similar $C^1$-smooth vorticity profile existed.

\begin{theorem}
\label{thm:two:fifths}
Assume the 3D Euler equation~\eqref{eq:Euler} has initial data $\omega_0 \in C^1(\mathbb{R}^3)$ with finite kinetic energy, i.e.~$u_0 = (-\Delta)^{-1} \nabla \times \omega_0 \in L^2(\mathbb{R}^3)$. Assume that the resulting local-in-time smooth solution blows up at some finite time $T_*>0$. If there exists $\gamma>0$ such that
\begin{equation}
\sup_{t\in [0,T_*)} (T_*-t)^{1+\gamma} \|\nabla \omega(\cdot,t)\|_{L^\infty(\mathbb{R}^3)} < \infty,
\label{eq:two:fifths:ass:1}
\end{equation}
then $\gamma \geq 2/5$. 
\end{theorem}
\begin{remark}
\label{rem:two:fifths}
A few comments concerning Theorem~\ref{thm:two:fifths} are in order:
\begin{itemize}[leftmargin=2em] 
\item[(a)] For a self-similar blowup (see~\eqref{eq:self-similar:intro}) with profile $\Omega$, assumption \eqref{eq:two:fifths:ass:1} is automatic if $\nabla \Omega \in L^\infty$.

\item[(b)] If in addition to~\eqref{eq:two:fifths:ass:1} we have that 
\begin{equation}
\label{eq:two:fifths:ass:2}
\sup_{t\in[0,T_*)} \|u(\cdot,t)\|_{L^p(\mathbb{R}^3)} < \infty,
\end{equation}
for some $p\geq 2$, 
then the lower bound on $\gamma$ becomes 
\[
\gamma \geq \frac{p}{p+3}.
\]
In particular, if $u \in L^\infty(0,T_*;L^3(\mathbb{R}^3))$, then $\gamma \geq 1/2$.
\end{itemize}
\end{remark}

\begin{proof}[Proof of Theorem~\ref{thm:two:fifths}]
We recall from~\cite{Constantin94a} (see also~\cite{ConstantinFefferman93,Constantin17})
that the magnitude of vorticity satisfies 
\begin{equation}
\partial_t |\omega| + u \cdot \nabla |\omega| = \alpha |\omega| \,,
\label{eq:vortex:stretching:classic}
\end{equation}
where the stretching factor $\alpha$ is defined as
\begin{equation}
\label{eq:vortex:stretching:alpha}
\alpha(x,t) := \frac{3}{4\pi} \text{p.v.} \int_{\mathbb{R}^3} \left(\hat{y} \cdot \xi(x,t)\right) \Bigl( \hat{y} \cdot \left(\omega(x+y,t) \times \xi(x,t)\right)\Bigr) \frac{dy}{|y|^3}
,
\end{equation}
where $\hat{y} = y/|y| \in \mathbb{S}^2$ and $\xi(x,t) = \omega(x,t)/|\omega(x,t)| \in \mathbb{S}^2$.

We consider a smooth cutoff function $\chi \colon \mathbb{R}_+ \to \mathbb{R}$, with $\chi \equiv 1$ on $[0,1]$, and $\chi \equiv 0$ on $[2,\infty)$. Let $R(t)>0$, to be determined later. 
We decompose $\alpha = \alpha_{\rm in} + \alpha_{\rm out}$, where 
\begin{align}
\alpha_{\rm in}(x,t) 
&:= \frac{3}{4\pi} \text{p.v.} \int_{\mathbb{R}^3} \left(\hat{y} \cdot \xi(x,t)\right)
\Bigl( \hat{y} \cdot \bigl(\omega(x+y,t) \times \xi(x,t)\bigr)\Bigr) \chi\Bigl(\frac{|y|}{R(t)}\Bigr) \frac{dy}{|y|^3}
\notag \\
&= \frac{3}{4\pi}  \int_{\mathbb{R}^3} \left(\hat{y} \cdot \xi(x,t)\right)
\Bigl( \hat{y} \cdot \bigl( (\omega(x+y,t)-\omega(x,t)) \times \xi(x,t)\bigr)\Bigr) \chi\Bigl(\frac{|y|}{R(t)}\Bigr) \frac{dy}{|y|^3}
,
\label{eq:alpha:in:def}
\end{align}
and
\begin{align}
\alpha_{\rm out}(x,t) 
&:= \frac{3}{4\pi} \text{p.v.} \int_{\mathbb{R}^3} \left(\hat{y} \cdot \xi(x,t)\right)
\Bigl( \hat{y} \cdot \bigl(\omega(x+y,t) \times \xi(x,t)\bigr)\Bigr)
\left( 1- \chi\Bigl(\frac{|y|}{R(t)}\Bigr) \right) \frac{dy}{|y|^3}
\notag \\
&= \frac{3}{4\pi} \int_{\mathbb{R}^3}  u(x+y,t) \cdot \Bigl( \bigl(5(\hat{y} \cdot \xi)^2 - 1\bigr) \hat{y} - 2  \bigl(\hat{y} \cdot \xi\bigr) \xi \Bigr) \left( 1-\chi\Bigl(\frac{|y|}{R(t)}\Bigr) \right)\frac{dy}{|y|^4}
 \notag\\
 &\qquad
 -\frac{3}{4\pi R(t)} \int_{\mathbb{R}^3}  \bigl(\hat{y} \cdot \xi\bigr) u(x+y,t) \cdot \left( \hat{y} \times \bigl(\xi \times \hat{y} \bigr) \right)\chi'\Bigl(\frac{|y|}{R(t)}\Bigr)  \frac{dy}{|y|^3}.
 \label{eq:alpha:out:def}
\end{align}
In the inner integral we used that $\omega(x,t) \times \xi(x,t) = 0$, while in the outer integral we have integrated by parts with respect to $\mathrm{curl}_y$.
We make two claims:
\begin{itemize}[leftmargin=2em]  
\item There exists a constant $C_{\rm in}>0$ depending only on $\chi$, such that 
\begin{equation}
|\alpha_{\rm in}(x,t)| \leq C_{\rm in} R(t) \|\nabla \omega(\cdot,t)\|_{L^\infty(B_{2 R(t)}(x))}. 
\label{eq:alpha:in:bound}
\end{equation}
This estimate follows directly from~\eqref{eq:alpha:in:def} upon noting that $|\hat{y}| = 1 = |\xi(x,t)|$, and by bounding $|\omega(x+y,t)-\omega(x,t)|\leq |y| \|\nabla \omega(\cdot,t)\|_{L^\infty(B_{2R(t)})}$.
\item For any $p\geq 1$, there exists a constant $C_{\rm out}>0$ depending only on $p$ and $\chi$, such that 
\begin{equation}
|\alpha_{\rm out}(x,t)| \leq C_{\rm out} R(t)^{-1-\frac{3}{p}} \|u(\cdot,t)\|_{L^p(\mathbb{R}^3)}. 
\label{eq:alpha:out:bound}
\end{equation}
This estimate follows directly from~\eqref{eq:alpha:out:def}, using H\"older's inequality and $|\hat{y}| = 1 = |\xi(x,t)|$.
\end{itemize}

The theorem now immediately follows from \eqref{eq:alpha:in:bound}--\eqref{eq:alpha:out:bound}, by letting  $p=2$. Indeed, classical solutions of the 3D Euler equations conserve their kinetic energy, $\|u(\cdot,t)\|_{L^2(\mathbb{R}^3)} = \|u_0\|_{L^2(\mathbb{R}^3)}$, and hence we have an a-priori bound for $\mathsf{RHS}_{\eqref{eq:alpha:out:bound}}$. Moreover, assumption~\eqref{eq:two:fifths:ass:1} implies an upper bound for $\mathsf{RHS}_{\eqref{eq:alpha:in:bound}}$. Optimizing in $R(t)$, and using~\eqref{eq:two:fifths:ass:1} we obtain the pointwise bound
\[
|\alpha(x,t)| \leq C (T_*-t)^{-\frac{5(1+\gamma)}{7}} 
\left( \sup_{t\in[0,T_*)} (T_*-t)^{1+\gamma} \|\nabla \omega(\cdot,t)\|_{L^\infty(B_{2R(t)}(x))} \right)^{\frac 57}  
 \|u_0\|_{L^2(\mathbb{R}^3)}^{\frac 27} ,
\]
where $C = C(C_{\rm in}, C_{\rm out}) >0$.
Thus, if $5(1+\gamma)/7 <1$---which is equivalent to $\gamma < 2/5$---then $\int_0^{T_*} \|\alpha(\cdot,t)\|_{L^\infty} dt < \infty$, and so by~\eqref{eq:vortex:stretching:classic} and the Beale-Kato-Majda criterion, no blowup can occur at time $T_*$. Therefore, a singularity necessitates $\gamma \geq 2/5$.

If in addition to~\eqref{eq:two:fifths:ass:1} we also know that~\eqref{eq:two:fifths:ass:2} holds, using~\eqref{eq:alpha:out:bound} we may analogously prove 
\[
|\alpha(x,t)| \leq C (T_*-t)^{-\frac{(p+3)(1+\gamma)}{2p+3}} 
\left(\sup_{t\in[0,T_*)} (T_*-t)^{1+\gamma} \|\nabla \omega(\cdot,t)\|_{L^\infty(B_{2R(t)}(x))} \right)^{\frac{p+3}{2p+3}}  
\!\!
\sup_{t\in[0,T_*)} \|u(\cdot,t)\|_{L^p(\mathbb{R}^3)}^{\frac{p}{2p+3}} 
,
\]
where $C = C(C_{\rm in},C_{\rm out},p)>0$.
Thus, if $(p+3)(1+\gamma)/(2p+3) < 1$---which is equivalent to $\gamma < p/(p+3)$---then $\int_0^{T_*} \|\alpha(\cdot,t)\|_{L^\infty} dt < \infty$, and no blowup can occur at time $T_*$. Therefore, a singularity necessitates $\gamma \geq p/(p+3)$, as claimed in item~(c) of Remark~\ref{rem:two:fifths}.
\end{proof}
 

\section{On globally self-similar blowup}
\label{sec:self-similar:Euler}

In order to obtain better lower bounds on the similarity exponent $\gamma \geq 2/5$, we analyze the hypothetical case in which the 3D incompressible Euler equations~\eqref{eq:Euler}, written in velocity form as  
\begin{equation}
\partial_t u + u \cdot \nabla u+ \nabla p = 0 , \qquad 
\nabla \cdot u =0  , 
\label{eq:Euler:velocity}
\end{equation}
admit a globally self-similar singularity at a time $ T_* > 0$, at a single point $x_* \in \mathbb{R}^3$. 

\subsection{The self-similar ansatz}
\label{sec:standing:SS:ass}
Due to Galilean symmetry, we assume without loss of generality that the singularity occurs at the space location $x_*=0$, and 
that the similarity profile for the velocity vanishes at this point.  Due to time-rescaling symmetry, we  take without loss of generality $T_*=1$. 
That is, we investigate the globally self-similar ansatz:
\begin{equation}
\begin{aligned}
&u(x,t) = (1-t)^{\gamma-1} U(y)
, 
\qquad
p(x,t) = (1-t)^{2(\gamma-1)} P(y) 
,
\\
&\omega(x,t) = \frac{1}{1-t} \Omega(y) 
, 
\qquad
y = \frac{x}{(1-t)^\gamma}
.
\end{aligned}
\label{eq:SS:ansatz}
\end{equation}
This ansatz is less restrictive than it seems. A more general form, $\omega (x,t) =
A(t)\widetilde \Omega\left(\fr{x-x(t)}{\ell(t)}, \tau (t)\right)$ leads to the form above. Indeed, writing the Euler equations in vorticity form and using the fact that the velocity and vorticity are related by the linear relation
$\omega = \na\times u$, we see that the nonlinear term $u\cdot\na\omega-\omega\cdot\na u$ scales with amplitude $A^2$ under the general ansatz above. We thus arrive at the ODE $\dot A = A^2$ where $\dot A$ is a time derivative {\footnote{ This time derivative can be with respect to an arbitrary internal clock, that is: $\fr{d}{dt}$ can be replaced by $\fr{1}{f(t)}\fr{d}{dt}$ where $f(t)$ is an arbitrary non-vanishing function.}}. This leads inevitably to $A(t) = (T-t)^{-1}$. The rest follows dividing the Euler equation by $A^2$ and taking the simplest situation, when the ensuing coefficients are constant. Then the constant coefficient in front of $y\cdot\na \Omega$ leads to $\ell(t) = \ell_0\left(1 -\fr{t}{T}\right)^{\gamma}$. The term due to $x(t)$ is integrated, because $A$ and $\ell$ are known. It follows that $x(t) = x_* - \fr{c}{\gamma}\ell(t)$ where $c$ is a constant, and this is dealt with by translating $\widetilde \Omega$.

\subsubsection{The stationary self-similar PDE}
The similarity profiles $U,P,\Omega$ appearing in~\eqref{eq:SS:ansatz} satisfy the (stationary) self-similar Euler equation in velocity form
\begin{equation}
(1-\gamma) U + \gamma (y \cdot \nabla) U + (U \cdot \nabla) U + \nabla P = 0,
\qquad 
\nabla \cdot U = 0 ,
\label{eq:profile:U}
\end{equation}
or equivalently, the (stationary) self-similar Euler equation in vorticity form
\begin{equation}
\Omega + \gamma (y \cdot \nabla) \Omega + (U \cdot \nabla) \Omega = (\Omega \cdot \nabla) U
,
\qquad 
\Omega = \nabla \times U
,
\qquad 
\nabla \cdot U = 0 .
\label{eq:profile:Omega}
\end{equation}
As mentioned earlier, due to Galilean symmetry (translation invariance), we assume throughout this section that the self-similar velocity profile $U$ vanishes at the origin. We also only consider profiles $U$ which are sub-linear at infinity.\footnote{
These are the profiles for which we may (typically) employ a localization or cutoff  procedure at
large values of $y$, to ensure that the velocity field $u$ in original $(x,t)$ coordinates has finite kinetic energy.}
We summarize these properties as
\begin{equation}
U(0) = 0,
\qquad
\lim_{|y|\to \infty} |y|^{-1} |U(y)| = 0 \,.
\label{eq:U:non:negotiable} 
\end{equation}

\subsubsection{Space rescaling and normalization}
Note that if $U$ and $\Omega$ are a solution of \eqref{eq:profile:Omega} on $\mathbb{R}^3$, then   
$U_{\lambda}(y)= \lambda U(\tfrac{y}{\lambda})$ and $\Omega_{\lambda}(y) =  \Omega(\tfrac{y}{\lambda})$, 
are also a solution of \eqref{eq:profile:Omega} on $\mathbb{R}^3$.
Under this rescaling, we have\footnote{As we shall see later (cf.~\eqref{eq:ass:medium:rare}), $\Omega$ is only expected to lie in $L^p(\mathbb{R}^3)$ when $p > 3 \gamma$.} 
\begin{equation}
\|\nabla^2 U_\lambda\|_{L^\infty(\mathbb{R}^3)}
= \tfrac{1}{\lambda} \|\nabla^2 U\|_{L^\infty(\mathbb{R}^3)}
,\quad
\|\nabla \Omega_\lambda\|_{L^\infty(\mathbb{R}^3)}    
= \tfrac{1}{\lambda}  \|\nabla \Omega\|_{L^\infty(\mathbb{R}^3)}    
,\quad 
\| \Omega_\lambda\|_{L^p(\mathbb{R}^3)}    
=  \lambda^{\frac{3}{p}} \|\Omega\|_{L^p(\mathbb{R}^3)}   
.
\label{eq:Omega:scaling}
\end{equation}
In order to fix units, throughout this section we consider profiles which are normalized via~\eqref{eq:Omega:scaling} to satisfy
\begin{equation}
\| \nabla^2 U \|_{L^\infty(\mathbb{R}^3)}  = 1 .
\label{eq:U:scaling} 
\end{equation}

\subsubsection{Behavior of the profiles at space infinity}
While the assumption of  sublinearity of $U$ at infinity (cf.~\eqref{eq:U:non:negotiable}) is necessary for any reasonable self-similar profile, the self-similar PDEs~\eqref{eq:profile:U} and~\eqref{eq:profile:Omega} impose more stringent assumptions. Indeed, \eqref{eq:U:non:negotiable} dictates that the term $\gamma y \cdot \nabla$ is stronger than the term $U \cdot \nabla$ for large values of $|y|$. Thus, the behavior of $U$ and $\Omega$ as $|y|\to \infty$ must respect the kernel of the operators $(1-\gamma) + \gamma y \cdot \nabla$ (for $U$) and $1 +\gamma y \cdot \nabla$ (for $\Omega$). This formally suggests that the leading order behavior of the self-similar profiles in the far field is given by 
\[
|U| \sim |y|^{\frac{\gamma-1}{\gamma}} ,
\qquad\mbox{and}\qquad
|\Omega| \sim |y|^{-\frac{1}{\gamma}},
\qquad\mbox{as} \qquad |y|\to \infty.
\]
It is convenient to quantify the above asymptotic descriptions. Recalling that $U(0) = 0$, we assume that there exists a constant $\mathsf{C}_\flat > 0 $ such that
\begin{equation}
  |U(y)| \leq \mathsf{C}_\flat |y| \langle y\rangle^{-\frac{1}{\gamma}},
  \qquad
  \mbox{and}
  \qquad
  |\Omega(y)| + |\nabla U(y)| \leq \mathsf{C}_\flat  \langle y\rangle^{-\frac{1}{\gamma}},
  \qquad
  \mbox{for all}
  \qquad y\in\mathbb{R}^3.
  \label{eq:ass:medium:rare}
\end{equation}

\subsection{A remark for the viscous case.}
We consider here hypothetical  blow up  for the 3D incompressible Navier-Stokes equation {(set $\mathsf{RHS}_{\eqref{eq:Euler:velocity}} = \Delta u$ instead of $0$)}, where the vorticity $\omega(x,t)$ is approximately self-similar.

\begin{proposition}\label{nse}
Let $\omega(x,t)$ be a solution of the 3D incompressible Navier-Stokes equation which can be written as
\begin{equation*}
\omega(x,t) 
= {\omega_{\rm in}(x,t)} + \omega_{\rm out}(x,t),
\end{equation*}
where $\omega_{\rm out}$ is regular, in the sense that
\begin{equation}
\int_0^{{1}} \|\omega_{\rm out}(\cdot,t)\|_{L^2(\Rr^3)}^4dt 
\le \Gamma<\infty,
\label{omegaoutfine}
\end{equation}
and where $\omega_{\rm in}$ is supported in $B_{\ell(t)}(x(t))$
 with $\ell(t)\sim (1-t)^{\gamma}$, {and satisfies}
\begin{equation}
\|\omega_{\rm in}(\cdot,t)\|_{L^2(B_{\ell(t)}(x(t)))} \le  C (1-t)^{-1+\fr{3\gamma}{2}}
\label{omegainbad}
\end{equation}
{for some $C<\infty$.}
Then, if the solution blows up at time $t=1$, it follows that $\gamma\leq \frac{1}{2}$.
\end{proposition}

\begin{remark}
The assumption \eqref{omegainbad} is satisfied if
\[
\omega_{\rm in}(x,t) \sim
\frac{1}{{1}-t}\Omega\left(\fr{x-x(t)}{({1}-t)^{\gamma}}\right) 
\qquad\mbox{in} \qquad B_{\ell(t)}(x(t)), 
\qquad\mbox{as} \qquad t \to 1^-,
\] 
with $|\Omega(y)| \le  C$ for $|y|\le 1$, {and $\ell(t) \sim (1-t)^\gamma$}. 
The assumption \eqref{omegaoutfine} is satisfied {for instance} if 
$\omega_{\rm out}(x,t) = 0$ for $|x-x(t)|\le \ell(t)$,
\[
|\omega_{\rm out}(x,t)| \le C_1 |x-x(t)|^{-\fr{1}{\gamma}}, 
\qquad \mbox{for} \qquad \ell(t) \leq |x-x(t)|\leq 1,
\]
or, more generally, if
\be
\int_0^1\|\omega_{\rm out}(t)\|^4_{L^2(B_1(x(t)))}dt 
\leq C_2 <\infty,
\la{outin}
\ee
and if, for instance
\[
\sup_{|x - x(t)|\ge 1}|\omega_{\rm out}(x,t)| \le C_3.
\] 
Indeed, because $\omega_{\rm out}(x,t) = \omega(x,t)$ for all $|x-x(t)|\ge 1$ we have  $|\omega_{\rm out}(x,t)|^2 \le C_3|\omega(x,t)|$ for 
$|x-x(t)|\ge 1$. Then $\|\omega_{\rm out}(t)\|_{L^2(\Rr^3\setminus B(x(t),1))}^2 \le C_3\|\omega(t)\|_{L^1(\Rr^3)}$. As it is well known, $\sup_t \|\omega (t)\|_{L^1(\Rr^3)}$ is bounded in terms of the initial data~\cite{ConstantinArea}, so the assumptions imply that $\omega_{\rm out}\in L^4(dt; L^2(\Rr^3))$. 

Thus, the assumption~\eqref{omegaoutfine}  is satisfied if $\omega$ is a smooth vorticity matched to an inner self-similar blow up ansatz with asymptotic behavior at infinity of the type~\eqref{eq:ass:medium:rare}. The main point is that both the inner blow up profile and the matching vorticity magnitude majorant  $z(x,t) = |x-x(t)|^{-\fr{1}{\gamma}}\mathbf 1_{\{\ell(t)\le |x-x(t)| \le 1\}}$ have the property  that  they belong to $ L^4(dt; L^2(\Rr^3))$ if $\gamma>\fr{1}{2}$, which is a class of non-blow up.
\end{remark}

\begin{proof}[{Proof of Proposition~\ref{nse}}]
We add the contributions {from the inner~\eqref{omegainbad} and outer~\eqref{omegaoutfine} pieces} and deduce
\begin{equation*}
\int_0^{{1}} \|\omega (\cdot,t)\|_{L^2(\Rr^3)}^4dt \le 16\left(\Gamma + \fr{C^4}{6\gamma-3}\right).
\end{equation*}
If $\gamma>\fr{1}{2}$, the right hand side is finite, and there is no blow up because of a well-known regularity criterion~\cite{CFbook}.
\end{proof}

\subsection{Self-similar vortex stretching}
We revisit aspects of the proof of Theorem~\ref{thm:two:fifths}, under the additional assumption that a globally self-similar profile exists (as specified in Section~\ref{sec:standing:SS:ass}).

Denote the vorticity direction vector as
\[
\Xi(y) = \frac{\Omega(y)}{|\Omega(y)|}
\,,
\]
and  define the self-similar vortex stretching factor by
\[
\mathsf{A}(y): = \Xi^i(y) \partial_i U^j(y) \Xi^j(y)
.
\]
Taking the dot product of \eqref{eq:profile:Omega}  with $\Omega |\Omega|^{p-2}$ we deduce the identity
\begin{equation}
|\Omega|^p +  \tfrac 1p (\gamma y + U ) \cdot \nabla (|\Omega|^p) = \mathsf{A} |\Omega|^p 
.
\label{eq:Omega:SS:Lp}
\end{equation}
In view of~\eqref{eq:ass:medium:rare} we expect $|\Omega|$ to be integrable at infinity only for $p > 3\gamma$, which is why we stated \eqref{eq:Omega:SS:Lp} with general $p$, as opposed to the standard $p=2$ formulation.

By the very definition of $\mathsf{A}$, the fact that $|\Xi(y)|=1$, and using assumption~\eqref{eq:ass:medium:rare} we have that $|\mathsf{A}(y)|\leq \mathsf{C}_\flat \langle y \rangle^{-\frac{1}{\gamma}}$ for all $y\in \mathbb{R}^3$. We aim to obtain a direct bound for the vortex stretching term; from~\eqref{eq:vortex:stretching:alpha}, we have that 
\begin{align}
\mathsf{A}(y)
= \frac{3}{4\pi} \text{p.v.} \int_{\mathbb{R}^3} \left(\hat{z} \cdot \Xi(y)\right) \left[\hat{z} \cdot \left(\Omega(y+z) \times \Xi(y)\right)\right] \frac{dz}{|z|^3}
\label{eq:stretching_SS}
,
\end{align}
where $\hat{z} = z/|z|$.  
Let $\chi = \chi(|z|)$ be a smooth radially decreasing cutoff function, with $\chi \equiv 1$ in $B_1$ and $\chi \equiv 0$ in $B_2^\complement$. Let $L>0$ be a length scale to be optimized later. Then, from~\eqref{eq:stretching_SS},  and \eqref{eq:ass:medium:rare},  we obtain that for $p > \max\{3\gamma,1\}$: 
\begin{align}
&\left| \mathsf{A}(y)
- \frac{3}{4\pi} \int_{\mathbb{R}^3} \left(\hat{z} \cdot \Xi(y)\right) \left[\hat{z} \cdot \left(  \Omega(y+z)   \times \Xi(y)\right)\right]  \chi\left( \tfrac{|z|}{L} \right) \frac{dz}{|z|^3}
\right|
\notag\\
&\qquad
\leq 
\frac{3^{1/p} (p-1)^{(p-1)/p} }{(4\pi)^{1/p}}  \|\Omega\|_{L^p(\mathbb{R}^3)}  L^{-\frac 3p}
.
\label{eq:stretching_SS_2p}
\end{align}
On the other hand,  since $\nabla \Omega \in L^\infty(\mathbb{R}^3)$, using the cancellation property $\Omega(y) \times \Xi(y) = 0$ and Taylor's theorem we obtain the estimate
\begin{align}
&\frac{3}{4\pi} \left| \int_{\mathbb{R}^3} \left(\hat{z} \cdot \Xi(y)\right) \left[\hat{z} \cdot \left( \bigl( \Omega(y+z) - \Omega(y) \bigr) \times \Xi(y)\right)\right] \chi\left( \tfrac{|z|}{L} \right) \frac{dz}{|z|^3}
\right|
\leq 
6 \|\nabla \Omega\|_{L^\infty(\mathbb{R}^3)}  L
.
\label{eq:stretching_SS_3p}
\end{align}
Optimizing the choice of $L$ in ~\eqref{eq:stretching_SS_2p} and~\eqref{eq:stretching_SS_3p}, we deduce the pointwise bound 
\begin{align}
|\mathsf{A}(y)| \leq  \mathsf{C}_p \|\nabla \Omega\|_{L^\infty(\mathbb{R}^3)}^{\frac{3}{p+3}} \|\Omega\|_{L^p(\mathbb{R}^3)}^{\frac{p}{p+3}}
,
\qquad
\mathsf{C}_p:= 2 \cdot 6^{\frac{3}{p+3}} \bigl( \tfrac{3 }{4\pi}\bigr)^{\frac{1}{p+3}}  
(p-1)^{\frac{p-1}{p+3}}
,
\label{eq:stretching_SS_4} 
\end{align}
for all $y\in \mathbb{R}^3$.
We note that the quantity present in \eqref{eq:stretching_SS_4}, namely
$\|\nabla \Omega\|_{L^\infty(\mathbb{R}^3)}^{\frac{3}{p+3}} \|\Omega\|_{L^p(\mathbb{R}^3)}^{\frac{p}{p+3}}$,  is scaling invariant under the natural scaling from~\eqref{eq:Omega:scaling}.

We wish to pair the pointwise upper bound for $\mathsf{A}$ in~\eqref{eq:stretching_SS_4} with a lower bound. In order to achieve this, we evaluate~\eqref{eq:Omega:SS:Lp} at a global maximum of $|\Omega|$, to directly obtain:
\begin{proposition}
\label{prop:A=1}
If $y_{*}$ is a (self-similar) space location at which  $|\Omega|\not \equiv 0$ attains its global maximum, then  $\mathsf{A}(y_{*}) = 1$. 
\end{proposition}

By combining the bound \eqref{eq:stretching_SS_4} with Proposition~\ref{prop:A=1}, we have thus established:
\begin{theorem}
\label{thm:Omega:not:small}
Let $(U,\Omega)$ be any nontrivial solution of~\eqref{eq:profile:Omega} which satisfies the bound~\eqref{eq:ass:medium:rare} for some $\gamma >0$.
Then, the vorticity profile cannot be too small, in the sense $\Omega$ must obey the scaling invariant lower bound
\begin{equation}
\|\nabla \Omega\|_{L^\infty(\mathbb{R}^3)}^{\frac{3}{p+3}} \|\Omega\|_{L^p(\mathbb{R}^3)}^{\frac{p}{p+3}}
\geq 
 \mathsf{C}_p^{-1}, 
\label{eq:Omega:not:small}
\end{equation}
for any $p > \max\{3 \gamma,1\}$, where the constant $\mathsf{C}_p>0$ is as defined in~\eqref{eq:stretching_SS_4}.
\end{theorem}
\begin{proof}[Proof of Theorem~\ref{thm:Omega:not:small}]
If $\nabla \Omega \not\in L^\infty(\mathbb{R}^3)$, there is nothing to prove; otherwise, since \eqref{eq:stretching_SS_4} must hold at ${y_{*}} = {\rm argmax}_{y\in \mathbb{R}^3} |\Omega|$, the bound \eqref{eq:Omega:not:small} follows from $\mathsf{A}({y_{*}}) = 1$, cf.~{Proposition}~\ref{prop:A=1}.
\end{proof}

\begin{remark}
Under the normalization~\eqref{eq:U:scaling}, {and} using the definition of $\mathsf{C}_p$, we obtain from~\eqref{eq:Omega:not:small} the requirement
\begin{equation}
\|\Omega\|_{L^p(\mathbb{R}^3)} 
\geq 2^{- 1 - \frac{3}{p}}
\bigl( \tfrac{\pi}{1296}\bigr)^{\frac{1}{p}}  
(p-1)^{-1 + \frac{1}{p}}  , 
\label{eq:Omega:not:small:normalized}
\end{equation}
whenever $p > \max\{3\gamma,1\}$.
\end{remark}

\subsection{Self-similar Lagrangian particle trajectories}
{It is convenient to denote the total transport velocity appearing in~\eqref{eq:profile:U} and~\eqref{eq:profile:Omega} as 
\be
V(y) := \gamma y + U(y).
\la{eq:V:def}
\ee
The self-similar Lagrangian trajectories are defined as solutions of the {autonomous} ODE
\begin{align}
\frac{d}{d\tau} Y(a,\tau) = V(Y(a,\tau))
\label{eq:Lagrangian:SS}
\end{align}
with $Y(a, 0) = a$.  Let $X = X(a,t)$ denote the Lagrangian particle trajectories associated to the velocity field $u = u(x,t)$; that is, $\frac{d}{dt} X(a,t) = u(X(a,t),t)$ for all $t \in (0,1)$, with initial datum $X(a,0) = a$. The self-similar Lagrangian particle trajectories are related to $X$ via
\be
Y(a,\tau) = \frac{X(a,t)}{(1-t)^\gamma} ,
\qquad
\mbox{with}
\qquad
\tau = - \log\left(1 - t \right) .
\la{YX}
\ee
We note that for $a=0$ we have $Y(0,\tau) = 0$ for all $\tau \in \mathbb{R}$; the origin is a fixed point of the dynamical system induced by~\eqref{eq:Lagrangian:SS}.

\subsubsection{The self-similar Cauchy formula and the Weber formula}
The Cauchy vorticity (vector) transport formula   $\omega(X(a,t),t) =  (\nabla_{a} X)(a,t)  \omega_0 (a)$ transformed into self-similar coordinates becomes
\begin{equation}
\Omega(Y(a,\tau)) = e^{-(\gamma+1)\tau}  (\nabla_{a}Y)(a,\tau) \Omega(a) ,
\label{eq:vorticity:transport}
\end{equation}
for all $a \in \mathbb{R}^3$ and all $\tau\in \Rr$.
The difference between~\eqref{eq:vorticity:transport} and the classical Cauchy formula is that $\mathrm{det}(\nabla_a Y)(a,\tau) = e^{3\gamma \tau}$ (instead of $=1$), because the vector field $\gamma y + U(y)$ is not divergence free; we have $ \nabla \cdot (\gamma y+U(y)) = 3\gamma$. 
The classical Weber formula for Euler equations is
\be
u^j(X(a,t),t)\pa_iX^j(a,t) - u^i(a,0) = \pa_i \pi(a,t)
\la{weber}
\ee
and is obtained by checking that $\pa_t((\na X(a,t))^T u(X(a,t),t))$ is a gradient. The self-similar Weber formula is
\be
e^{(1-2\gamma)\tau} U^j(Y(a,\tau))\pa_i Y^j(a,\tau) - U^i(a) = \pa_i q(a,\tau).
\la{web}
\ee
for an appropriate scalar function $q$. The proof of the Weber formula \eqref{web}  follows either from changing variables in \eqref{weber}, or directly in similarity variables from the observation stemming from \eqref{eq:profile:U}
that
\begin{equation*}
\pa_{\tau}( (\na Y)^TU(Y)) + (1-2\gamma) (\na Y)^TU(Y) + \na\left (P(Y) - \fr{|U(Y)|^2}{2}\right) =0
\end{equation*}
holds at each $a$. Thus,
\begin{equation*}
e^{(1-2\gamma)\tau} ((\na Y)^TU(Y))(a,\tau) - U(a) = -\na  \int_0^{\tau} e^{(1-2\gamma)s}\left (P(Y) - \fr{|U(Y)|^2}{2}\right)ds
\end{equation*}
is a gradient.

\subsubsection{The self-similar Kelvin circulation theorem}

Let $C(0) \subset \mathbb{R}^3$ be a simple closed curve parametrized as $C(0) = \{ a(\lambda) \colon \lambda \in \mathbb{T}\}$.
The circulation of $U$ on a loop $C(0)$ is the integral of the 1-form $U(y)dy$ on that loop, which is computed as
\be
\Gamma_{\rm{ss}}(0) = \oint_{C(0)} Udy = \int_{\mathbb T} U^j(a(\lambda))\dot{a}^j(\lambda)d\lambda,
\la{Gammazero}
\ee
where we have denoted $\dot{a}^j = \frac{d}{d\lambda} a^j$. We note that  $V$ (recall~\eqref{eq:V:def}) and $U$ have the same circulation on any loop because the 1-forms $V(y)dy$ and $U(y)dy$ differ by an exact 1-form, $d\fr{\gamma |y|^2}{2}$.

The self-similar Weber formula \eqref{web}  says that the pull back of the 1-form $e^{(1-2\gamma)\tau}U(y)dy$ under the diffeomorphism $Y(a,\tau)$ is the sum of the form $U(a)da$ and an exact 1-form, $dq$. The circulation of $U$ on $C(\tau)$ where $C(\tau) = Y(C(0),\tau)$ is  the  push forward (image) of the loop $C(0)$ is
\be
\Gamma_{\rm{ss}}(\tau) = \oint_{C(\tau)} U(y)dy.
\la{circss}
\ee
The self-similar Weber formula \eqref{web} implies
\be
e^{(1-2\gamma)\tau}\oint_{C(\tau)} U(y)dy = \oint_{C(0)} U(y)dy.
\la{circinv}
\ee
Indeed, multiplying \eqref{web} by $\dot{a}^i(\lambda)$ and integrating $d\lambda$ on $\mathbb T$  we have
\begin{align}
e^{(1-2\gamma)\tau}\Gamma_{\rm{ss}}(\tau) 
&= e^{(1-2\gamma)\tau}\oint_{C(\tau)}U(y)dy \notag\\
&= e^{(1-2\gamma)\tau}\int_{\mathbb{T}}  U^j(Y(a(\lambda),\tau)) \; (\partial_{i}Y^j)(a(\lambda),\tau)  \dot{a}^i(\lambda) d \lambda
\notag \\
&= \Gamma_{\rm{ss}}(0).
\label{eq:Kelvin:SS}
\end{align}

In original variables, the Weber formula  says that the pull back  of the 1-form $u(x,t)dx$ under the Lagrangian flow $X(a,t)$ is the sum of the time independent 1-form $u(a,0)da$ and an exact 1-form.  As a consequence, Kelvin's circulation theorem states that the circulation $\Gamma(t)$ around the material loop $C(t) = \{ X(a(\lambda),t) \colon \lambda \in \mathbb{T}\}$ is conserved. 
That is, we  have $\frac{d}{dt} \Gamma(t) = 0$, where 
\begin{equation*}
\Gamma(t) 
= \oint_{C(t)} u(x,t)dx
= \int_{\mathbb{T}} u^j(X(a(\lambda),t),t) (\partial_{i}X^j)(a(\lambda),t)  \dot{a}^i(\lambda) d \lambda .
\end{equation*}
Changing  to self-similar variables using \eqref{YX}  and time $\tau=-\log(1-t)$ we note that 
\[
\Gamma(t) = e^{(1-2\gamma)\tau}\Gamma_{\rm ss}(\tau)
.
\]

\begin{remark}
Because $e^{(1-2\gamma)\tau}\Gamma_{\rm ss}(\tau)$ is constant in time, if the integral $\int_{\mathbb{T}} (\ldots) d\lambda$ appearing in~\eqref{eq:Kelvin:SS} is nonzero, remains bounded, and  bounded away from zero as $\tau \to \infty$, then we must have $\gamma = 1/2$ in order to avoid a vanishing/blowing up exponential factor. Thus, the self-similar Kelvin circulation theorem identifies the similarity exponent $\gamma=1/2$ as being distinguished for 3D Euler, without knowledge of any viscous regularization. It also indicates that fixed points of~\eqref{eq:Lagrangian:SS}, and by extension points on the stable manifold of these fixed points, play an important role in the analysis. A result in this direction is offered in Theorem~\ref{thm:axissym:with:swirl} below.
\end{remark}

\subsubsection{{The self-similar Bernoulli function}}
Note that $\nabla \times V = \nabla \times U = \Omega$ since $\nabla \times y = 0$.
In~\eqref{eq:profile:U} we rewrite the first three terms as 
\begin{align*}
(1-\gamma) U +  (\gamma y + U) \cdot \nabla U 
&= (1-\gamma) (V - \gamma y) - \gamma V  + V \cdot \nabla V \\
&=- \gamma (1-\gamma) \nabla \frac{|y|^2}{2} 
+ (1- 2\gamma) V  
+ \Omega \times V 
+ \nabla \frac{|V|^2}{2}
.
\end{align*}
After rearranging, we obtain that~\eqref{eq:profile:U} is equivalent to
\begin{equation}
\label{eq:Bernoulli:SS}
(1-2\gamma)\, V + \Omega \times V + \nabla \mathcal{H} = 0,
\end{equation}
where the \emph{self-similar Bernoulli function}\footnote{See~\cite{NRS96} for use of the analogous Bernoulli function in the case of  3D incompressible Navier-Stokes.} is defined by
\begin{equation}
\label{eq:Bernoulli:def}
\mathcal{H}(y) := \frac{1}{2} |\gamma y + U(y)|^2 + P(y) + \frac{\gamma(\gamma-1)}{2} |y|^2 .
\end{equation}
Dotting~\eqref{eq:Bernoulli:SS} with $V$ yields the transport identity 
\begin{equation}
\label{eq:Bernoulli:transport}
V \cdot \nabla \mathcal{H} = (2\gamma - 1)|V|^2,
\end{equation}
which, restated along self-similar Lagrangian trajectories~\eqref{eq:Lagrangian:SS} is 
\begin{equation}
\label{eq:Bernoulli:transport:2}
\frac{d}{d\tau} \mathcal{H}(Y(a,\tau)) = (2\gamma-1) |V(Y(a,\tau))|^2
.
\end{equation}
Two immediate consequences follow. First, for $\gamma < 1/2$ the Bernoulli function is \emph{strictly decreasing} along any non-stationary trajectory:
\begin{equation}
\label{eq:Bernoulli:Lyapunov}
\gamma < \frac 12 
\qquad \Longrightarrow \qquad
\frac{d}{d\tau} \mathcal{H}(Y(a,\tau)) \leq 0,
\end{equation}
with equality if and only if $V(Y(a,\tau)) = 0$.
Second, since $|V(y)| \sim \gamma |y|$ for $|y| \gg 1$ (cf.~\eqref{eq:ass:medium:rare}) and for $\gamma <1$ we have $P(y) = \mathsf{o}(1)$ as $|y|\to \infty$,\footnote{A bound for the  pressure may be obtained as in~\cite{NRS96}. Taking the divergence of~\eqref{eq:profile:U} we have that $P = P_h + \mathcal{R}_i \mathcal{R}_j(U^i U^j)$, where $P_h$ is harmonic and $\mathcal{R}$ is the vector of Riesz-transforms. For $\gamma<1$, the second term in this expression decays to $0$ as $|y|\to \infty$. From~\eqref{eq:profile:U} and~\eqref{eq:ass:medium:rare} we also have that $|\nabla P|$ (and hence also $|\nabla P_h|$) decays to $0$ as $|y| \to \infty$ when $\gamma<1$, showing that $P_h$ is a constant. Since the pressure is defined only up to a constant, $P_h \equiv 0$. } the far-field behavior of the Bernoulli function is
\begin{equation}
\label{eq:Bernoulli:far:field}
\mathcal{H}(y) = \frac{\gamma(2\gamma-1)}{2} |y|^2 + \mathsf{o}(|y|^2)
\qquad \mbox{as } |y|\to\infty.
\end{equation}
In particular, for $\gamma < 1/2$ we have $\mathcal{H}(y) \to -\infty$ as $|y|\to \infty$.
}

\subsection{An outgoing property?}
\label{sec:outgoing?}
We note that all trajectories with sufficiently large labels escape to infinity. Indeed, the bound~\eqref{eq:ass:medium:rare} implies that for all $y \in \mathbb{R}^3$ such that  
\begin{equation}
\label{eq:R:flat:def}
|y| \geq R_\flat:= \max\left\{ 1, \Bigl(\frac{2\mathsf{C}_\flat}{\gamma}\Bigr)^\gamma \right\},
\end{equation}
we have $|U(y) \cdot y| \leq \fr{\gamma}{2} |y|^2$, and therefore $(\gamma y + U(y)) \cdot y \geq \frac{\gamma}{2} |y|^2$.   From~\eqref{eq:Lagrangian:SS} we then immediately obtain that if $|a| \geq  R_\flat$, then $|Y(a,\tau)|$ is a strictly increasing function of $\tau$, and we have the bounds
$|a| e^{\frac{\gamma}{2}\tau} 
\leq |Y(a,\tau)| \leq 
|a| e^{\frac{3\gamma}{2}\tau}$, and $
e^{\frac{\gamma}{2}\tau}
\leq |\nabla_a Y(a,\tau)| \leq 
e^{\frac{3\gamma}{2}\tau}$, 
for all $\tau\geq 0$. 

The exponential escape to infinity of self-similar Lagrangian flow maps (as exhibited above) is well-known to play a fundamental role, both in proving the existence and the stability of self-similar profiles. In the context of the compressible Euler equations, see~\cite{BSV23,MRRS22a,BCG25,CCSV24}.
In our incompressible 3D Euler setting we only know that trajectories $Y(a,\tau)$ for which $|a|$ is sufficiently large escape to infinity, and that the trajectory emanating from $a=0$ is frozen at the origin. 

A priori, we cannot rule out the existence of points $y \in \mathbb{R}^3 \setminus\{ 0 \}$ for which $V(y) = \gamma y + U(y) = 0$. 
In the compressible setting, such points exist---\textit{sonic points}---and they are well known to cause tremendous headaches, such as severe instabilities, even within the class of radially symmetric solutions~\cite{Guderley42,MRRS22a,BCG25}.
For incompressible 3D Euler equations, a 
\textit{global outgoing property}, quantified as a lower bound of the type 
\begin{equation}
\label{eq:outgoing:tarek}
V(y)\cdot y = (\gamma y + U(y) ) \cdot y \geq c_* |y|^2, 
\qquad
\mbox{for all }   y \in \mathbb{R}^3,
\qquad
\mbox{for some }
c_* > 0,
\end{equation}
has been recently identified by Elgindi~\cite{Elgindi25} as one of the two fundamental assumptions  needed to prove the existence of solutions to~\eqref{eq:profile:Omega}; see~\cite[Hypothesis 6.7]{Elgindi25}. Assumption~\eqref{eq:outgoing:tarek} implies that the only zero of $V$ occurs at $y=0$ (since $c_*>0$). In this section we want to allow for a generalized outgoing property, which allows for the degenerate case $c_* = 0$, allows the function $V$ to vanish away from the origin, and only makes assumptions that are local, near the nodal set of $V$.

\begin{definition}
\label{def:outgoing}
We say that the field $V(y) = \gamma y + U(y)$ satisfies the \emph{local outgoing property} if the nodal set
\[
\mathcal{N}_V := \{ y_* \in \mathbb{R}^3 \colon V(y_*)  = 0\}
\]
is finite,\footnote{As discussed below~\eqref{eq:R:flat:def}, the bound~\eqref{eq:ass:medium:rare} implies $\mathcal{N}_V \subset B_{R_\flat}(0)$. Hence, if we assume that $\mathcal{N}_V$ has \textit{no accumulation points}, then this set must automatically be finite.} and if there exists $c_* \geq 0$ and $\varepsilon_*>0$ such that for all $y_* \in \mathcal{N}_V$ we have
\begin{equation}
\label{eq:outgoing}
V(y) \cdot (y-y_*) \geq c_* |y-y_*|^2,
\qquad
\mbox{when}
\qquad |y-y_*|\leq \varepsilon_*.
\end{equation}
\end{definition}
A few remarks concerning Definition~\ref{def:outgoing} are in order:
\begin{itemize}[leftmargin=2em]
\item Note that $\mathcal{N}_V \neq \emptyset$ because $0 \in \mathcal{N}_V$ always. 
\item The global outgoing property~\eqref{eq:outgoing:tarek} implies the local outgoing property of Definition~\ref{def:outgoing}. Indeed, upon evaluating~\eqref{eq:outgoing:tarek}  at any  $y_* \in \mathcal{N}_V$ we deduce that $0 = V(y_*) \cdot y_* \geq c_* |y_*|^2$; this shows $y_*=0$ since $c_*>0$. Thus $\mathcal{N}_V = \{0\}$. Condition~\eqref{eq:outgoing} thus holds for any $\varepsilon_* >0$. 
\item When $\mathcal{N}_V = \{0\}$,~\eqref{eq:outgoing} generalizes~\eqref{eq:outgoing:tarek} not just in the local-in-$y$ nature of the inequality; when $c_*=0$ it also allows for $V(y) \cdot y = \gamma |y|^2 + y \cdot U(y)$ to vanish as $|y|^{2n}$ for $n\geq 2$ as $|y|\to 0$.
\end{itemize}

Our next result points out a significant consequence of assumption~\eqref{eq:outgoing}.
\begin{theorem}
\label{thm:outgoing:one:half}
Let $U$ be a $C^2$ smooth globally self-similar velocity profile for 3D incompressible Euler. If there exists $y_* \in \mathcal{N}_V$  such that $\Omega(y_*)\neq 0$, and if  the local outgoing property~\eqref{eq:outgoing} holds \emph{in an open neighborhood} of $y=y_*$ for some $c_* \geq0$, then
\[
\gamma \geq \frac 12 + c_*.
\]
\end{theorem}
\begin{proof}[Proof of Theorem~\ref{thm:outgoing:one:half}]
We rewrite the self-similar vorticity equation~\eqref{eq:profile:Omega} as
\begin{equation}
\label{eq:Omega:SS}
\Omega + V \cdot \nabla \Omega = \mathbb{S} \Omega,
\end{equation}
where  $\mathbb{S} = \frac 12 ( (\nabla U) + (\nabla U)^\intercal)$ is the rate of strain matrix.  If $y_* \in \mathbb{R}^3$ is a zero of $V$, restricting~\eqref{eq:Omega:SS} to $y=y_*$,  we deduce that the matrix
\[
\mathbb{S}_{y_*} := \mathbb{S}(y_*)
\] 
has $1$ as an eigenvalue, with associated eigenvector $\Omega(y_*) \neq 0$. 
Now $\mathbb{S}_{y_*}$ is a real, symmetric, traceless matrix ($\nabla \cdot U = 0$) and we denote its  eigenvalues (which are real) as $1,\lambda_2, \lambda_3$, with $\lambda_2 + \lambda_3 = -1$.

Since $V(y_*) = 0$, for $|y-y_*|\ll 1$ we  have 
\begin{align*}
V(y) \cdot (y-y_*)
&=
(y-y_*) \otimes (y-y_*) \colon \! (\nabla V)(y_*) 
+
\mathcal{O}(|y-y_*|^3)
\notag\\
&= 
(y-y_*) \otimes (y-y_*) \colon \!  ( \gamma \mathrm{Id} + \mathbb{S}_{y_*} ) 
+ 
\mathcal{O}(|y-y_*|^3) .
\end{align*} 
Thus, if the outgoing condition~\eqref{eq:outgoing} holds in an open neighborhood of $y_*$, then for all $|y-y_*|\ll1$ we have 
$\gamma + \hat{z} \cdot \mathbb{S}_{y_*} \hat{z} + \mathcal{O}(|y-y_*|) \geq c_*$, where $\hat{z} = \frac{y-y_*}{|y-y_*|}$. Passing $y \to y_*$, we deduce that the smallest eigenvalue of the matrix $\mathbb{S}_{y_*}$, satisfies $\lambda_{\rm min}(\mathbb{S}_{y_*}) \geq c_* - \gamma$. Thus, the sum of the two-smallest eigenvalues of $\mathbb{S}_{y_*}$ is $\geq 2 (c_* - \gamma)$, and due to the zero trace property, the largest eigenvalue of the matrix $\mathbb{S}_{y_*}$ satisfies $\lambda_{\rm max}(\mathbb{S}_{y_*}) \leq 2(\gamma-c_*)$. We have thus shown that the eigenvalues of $\mathbb{S}_{y_*}$ satisfy
\[
[\lambda_{\rm min}(\mathbb{S}_{y_*}),\lambda_{\rm max}(\mathbb{S}_{y_*})]\subseteq [c_* - \gamma, 2(\gamma -c_*)].
\] 

In particular, because the eigenvalue $1$ lies in this interval,  we deduce $1 \leq 2 (\gamma-c_*)$. The claimed lower bound $\gamma \geq c_* + 1/2$ now follows.
\end{proof}

Our next result considers the complementary case, when the vorticity vanishes on $\mathcal{N}_V$. 

\begin{theorem}
\label{thm:outgoing:global:new}
Let $U$ be a nontrivial $C^2$ smooth globally self-similar velocity profile for 3D incompressible Euler equations satisfying~\eqref{eq:ass:medium:rare}, normalized with~\eqref{eq:U:scaling}. 
Assume that the local outgoing property of Definition~\ref{def:outgoing} holds. Then, 
\[
\gamma \geq \frac12.
\]
\end{theorem}

\begin{proof}[Proof of Theorem~\ref{thm:outgoing:global:new}]
By contradiction, assume that $\gamma < 1/2$. 
We enumerate the nodal set of $V$ as $\mathcal{N}_V = \{ y_*^{(\ell)} \colon 0\leq \ell  \leq N\}$, for some $N\in \mathbb{N}$ and define (we use the convention $\min_{\emptyset} = + \infty$)
\begin{equation}
\delta_* := \frac16  \min\left( \frac 12 - \gamma , \min_{0\leq k\neq \ell \leq N} \bigl|y_*^{(\ell)} - y_*^{(k)}\bigr| \right) >0. 
\label{eq:delta:star:def}
\end{equation}
We also define the sets
\begin{equation*}
\mathcal{D} := \bigcup\limits_{\ell = 0}^{N} B_{\delta_*}\bigl(y_*^{(\ell)}\bigr),
\qquad
\mathcal{K}: = \mathcal{D}^\complement \cap \overline{B_{R_\flat}(0)},
\end{equation*}
where $R_\flat$ is as defined in~\eqref{eq:R:flat:def}. The set $\mathcal{K}$ is compact.

We claim that for all $y_0 \in \mathbb{R}^3$ there exists a unique $y_*=y_*(y_0) \in \mathcal{N}_V$ and a maximal time $\tau_0 = \tau_0(y_0)\leq 0$ such that the backwards-in-time flow map satisfies
\begin{equation}
Y(y_0,\tau) \in \overline{B_{2\delta_*}(y_*)}, \qquad \mbox{for all} \qquad \tau \leq  \tau_0.
\label{eq:backwards:to:check:new}
\end{equation}
In order to prove this claim, we first note that due to the discussion in the first paragraph of Section~\ref{sec:outgoing?}, there exists\footnote{This is because when $|y_0| > R_\flat$, we have that $|Y(y_0,\tau)| \leq |y_0| e^{\gamma \tau/2}$, as long as $|Y(y_0,\tau)| \geq R_\flat$. Thus, there exists a finite, sufficiently negative $\tau_0^\prime$ such that $|Y(y_0,\tau_0^\prime)| = R_\flat$. Once a backwards-in-time trajectory has entered the compact set $\overline{B_{R_\flat}(0)}$, it cannot leave it, by the very definition of $R_\flat$.} $\tau_0^\prime\leq 0$ such that $Y(y_0,\tau) \in \overline{B_{R_\flat}(0)}$ for all $\tau \leq \tau_0^\prime$. Define the set of times 
\[
J:= \{ \tau \in (-\infty,\tau_0^\prime] \colon Y(y_0,\tau) \in \mathcal{K} \}.
\]
Since all the zeros of $V(y) = \gamma y + U(y)$ lie in the interior of $\mathcal{D}$,  $V$ is continuous, and $\mathcal{K}$ is compact, there exists $c_\flat \in (0,1)$ such that 
\[
c_\flat \leq |V(y)| \leq c_\flat^{-1},
\qquad
\mbox{for all}
\qquad
y\in \mathcal{K}.
\]
Therefore, since $Y(y_0,\tau) \in \overline{B_{R_\flat}(0)}$ for all $\tau \leq \tau_0^\prime$, since the Bernoulli function $\mathcal{H}$ is continuous and $\gamma<1/2$, upon integrating~\eqref{eq:Bernoulli:transport:2} on $[\tau,\tau_0^\prime]$, by the definition of the set $J$ we obtain
\begin{align*}
(1-2\gamma) c_\flat^2 \bigl| J \cap [\tau,\tau_0^\prime] \bigr| 
&\leq  (1-2\gamma)  \int_{\tau}^{\tau_0^\prime} {\bf 1}_{\tau^\prime \in J} \bigl|V(Y(y_0,\tau^\prime))\bigr|^2 \, d\tau^\prime
\\
&\leq 
(1-2\gamma) \int_{\tau}^{\tau_0^\prime} \bigl|V(Y(y_0,\tau^\prime))\bigr|^2 \, d\tau^\prime
\\
&= \mathcal{H}(Y(y_0,\tau)) - \mathcal{H}(Y(y_0,\tau_0^\prime))
\leq \max_{\overline{B_{R_\flat}(0)}} \mathcal{H} - \min_{\overline{B_{R_\flat}(0)}} \mathcal{H} 
=: \mathsf{H}_\flat < \infty,
\end{align*}
for any $\tau \leq \tau_0^\prime$. Letting $\tau \to -\infty$, we deduce that the set $J$ has finite Lebesgue measure, with $|J| \leq \mathsf{H}_\flat c_\flat^{-2} (1-2\gamma)^{-1}$.
Thus, the trajectory $\{ Y(y_0,\tau) \}_{\tau \leq \tau_0^\prime}$ spends a finite amount of time in $\mathcal{K}$. Note however that this trajectory spends all time, i.e.~the entire interval $(-\infty,\tau_0^\prime]$, in $\overline{B_{R_\flat}(0)}$.

Next, for each $0 \leq \ell \leq N$ we consider the closed annulus $\mathcal{A}_\ell := \{ y \in \mathbb{R}^3 \colon \delta_* \leq |y-y_*^{(\ell)}| \leq 2\delta_* \}$. By the definition of $\delta_*$ in~\eqref{eq:delta:star:def}, for every $y \in \mathcal{A}_\ell$ and every $k \neq \ell$ we have $|y - y_*^{(k)}| \geq |y_*^{(k)} - y_*^{(\ell)}| - |y - y_*^{(\ell)}| \geq 6\delta_* - 2\delta_* = 4\delta_* > \delta_*$, and therefore $\mathcal{A}_\ell \subset \mathcal{D}^\complement$. In particular, if $\tau^\prime \leq \tau_0^\prime$ is such that $Y(y_0,\tau^\prime) \in \mathcal{A}_\ell$, then $\tau^\prime \in J$. Let us call a {\em crossing of $\mathcal{A}_\ell$} a time interval $[\tau_2,\tau_1] \subset J$ such that $Y(y_0,\tau^\prime) \in \mathcal{A}_\ell$ for all $\tau^\prime \in [\tau_2,\tau_1]$, and such that $\{ |Y(y_0,\tau_2) - y_*^{(\ell)}| , |Y(y_0,\tau_1) - y_*^{(\ell)}| \} = \{\delta_*, 2\delta_*\}$. During a crossing $[\tau_2,\tau_1]$, the trajectory travels a distance of at least $\delta_*$ while remaining in $\mathcal{K}$, and hence $\tau_1 - \tau_2 \geq \delta_* c_\flat$. Due to our previously established bound on $|J|$, the total number of possible crossings of the annuli $\{\mathcal{A}_\ell\}_{\ell=0}^{N}$ with pairwise disjoint interiors has cardinality at most $\mathsf{H}_\flat (1-2\gamma)^{-1} c_\flat^{-3} \delta_*^{-1}$. Therefore, there exists $\tau_0^{\prime\prime} \leq \tau_0^\prime$ such that no crossing of any annulus $\{\mathcal{A}_\ell\}_{\ell=0}^{N}$ is possible on the time interval $(-\infty,\tau_0^{\prime\prime}]$.

Since the trajectory $\{ Y(y_0,\tau) \}_{\tau \leq \tau_0^{\prime\prime}}$ only spends a finite amount of time in $\mathcal{K}$, there exists $\tau_0 \leq \tau_0^{\prime\prime}$ such that $Y(y_0,\tau_0) \in \mathcal{D}$; say $Y(y_0,\tau_0) \in B_{\delta_*}\bigl(y_*^{(\ell)}\bigr)$ for some $0 \leq \ell \leq N$. We let $y_* := y_*^{(\ell)}$ and claim that $Y(y_0,\tau) \in \overline{B_{2\delta_*}(y_*)}$ for all $\tau \leq \tau_0$. Otherwise, by the intermediate value theorem this continuous trajectory would need to cross the annulus $\mathcal{A}_\ell$ on a time interval $[\tau_2,\tau_1] \subset (-\infty,\tau_0] \subset (-\infty,\tau_0^{\prime\prime}]$, a contradiction. 

Finally, \eqref{eq:delta:star:def} implies that distinct points of $\mathcal{N}_V$ lie at distance at least $6\delta_*$ from one another, and hence the point $y_*$ is uniquely determined. Since $Y(y_0,\cdot)$ is continuous and $\overline{B_{2\delta_*}(y_*)}$ is closed, the supremum of the admissible times $\tau_0$ is attained, so that $\tau_0(y_0)$ may indeed be taken to be maximal. This proves~\eqref{eq:backwards:to:check:new}.

Next, we estimate the size of $|(\nabla_a Y)(y_0,\tau)|$ for all $\tau \leq\tau_0=\tau_0(y_0)$. Fix $y_*= y_*(y_0)$ as in~\eqref{eq:backwards:to:check:new}. 
Differentiating~\eqref{eq:Lagrangian:SS} with respect to labels, and then contracting with $\nabla_a Y$, we obtain
\begin{equation*}
\frac 12 \frac{d}{d\tau} |\nabla_a Y(y_0,\tau)|^2 
=  
\gamma |\nabla_a Y(y_0,\tau)|^2 
+
\partial_{a_j}Y^i(y_0,\tau) \; \partial_{a_j}Y^k(y_0,\tau) \; (\partial_k U^i)(Y(y_0,\tau)).
\end{equation*}
For $\tau \leq \tau_0$ we have $|Y(y_0,\tau) - y_*|\leq 2 \delta_*$, and thus by the mean value theorem and the normalization $\|\nabla^2 U\|_{L^\infty} = 1$ (cf.~\eqref{eq:U:scaling}) we deduce 
\[
\bigl| (\nabla U)(Y(y_0,\tau)) - (\nabla U)(y_*)\bigr| \leq 2 \delta_*, 
\qquad \mbox{for all} \qquad
\tau \leq \tau_0.
\]
Since we have assumed $\gamma < 1/2$, by Theorem~\ref{thm:outgoing:one:half}, for every $y_* \in \mathcal{N}_V$, we must have $\Omega(y_*) =0$, and hence  $\nabla U(y_*) =  \mathbb{S}_{y_*}$,
where we recall that $\mathbb{S}_{y_*} = \frac 12 ( (\nabla U)(y_*) + (\nabla U)^\intercal(y_*))$ is the symmetric part of $\nabla U$ evaluated at $y_*$. Moreover, as in the proof of Theorem~\ref{thm:outgoing:one:half}, from the local outgoing property at $y_*$ and $c_* \geq 0$ we deduce that $[\lambda_{\rm min}(\mathbb{S}_{y_*}),\lambda_{\rm max}(\mathbb{S}_{y_*})]\subseteq [c_* - \gamma, 2(\gamma -c_*)]\subseteq [-\gamma,2\gamma]$. Combining the above estimates, we deduce that 
\begin{align*}
\frac 12 \frac{d}{d\tau} |\nabla_a Y(y_0,\tau)|^2 
&\geq  
\gamma |\nabla_a Y(y_0,\tau)|^2 
+
\partial_{a_j}Y^i(y_0,\tau) \; \partial_{a_j}Y^k(y_0,\tau) \; \mathbb{S}_{y_*}^{ik}
- 
2 \delta_* |\nabla_a Y(y_0,\tau)|^2 
\\
&\geq 
- 
2 \delta_* |\nabla_a Y(y_0,\tau)|^2,
\end{align*}
for all $\tau \leq \tau_0$. Integrating this inequality in time, it follows that 
\begin{equation}
\label{eq:slow:growth:Lag:grad}
\bigl|(\nabla_a Y) (y_0,\tau)\bigr|
\leq 
\bigl|(\nabla_a Y) (y_0,\tau_0)\bigr| e^{2 \delta_*(\tau_0-\tau)},
\qquad\mbox{for all}\qquad
\tau\leq \tau_0.
\end{equation}

To conclude the proof, we appeal to the self-similar Cauchy formula~\eqref{eq:vorticity:transport}, which states that 
\[
e^{(\gamma+1)\tau} \Omega(Y(y_0,\tau)) =  (\nabla_a Y)(y_0,\tau) \Omega(y_0)
\]
for all $y_0 \in \mathbb{R}^3$ and all $\tau \in \mathbb{R}$. We multiply from the left with the inverse matrix $(\nabla_a Y)^{-1}(y_0,\tau)$ and deduce
\[
\Omega(y_0) 
= 
e^{(\gamma+1)\tau} (\nabla_a Y)^{-1}(y_0,\tau) \Omega(Y(y_0,\tau)) 
.
\]
Noting that $\mathrm{det}(\nabla_a Y)(y_0,\tau) = e^{3\gamma \tau}\neq 0$ (so that this matrix is indeed invertible), and using that for a $3\times3$ matrix the cofactor matrix is quadratic in the entries of the original matrix, we obtain that 
\[
\bigl|(\nabla_a Y)^{-1}(y_0,\tau)\bigr|
\leq 18 e^{-3\gamma \tau}\bigl|(\nabla_a Y)(y_0,\tau)\bigr|^2
\leq 18 e^{-3\gamma \tau}\bigl|(\nabla_a Y) (y_0,\tau_0)\bigr|^2 e^{4 \delta_*(\tau_0-\tau)}.
\]
Moreover, since $\Omega(y_*)=0$, by the mean value theorem and the bound $\|\nabla  \Omega\|_{L^\infty} \leq 2$ (which follows from the normalization~\eqref{eq:U:scaling}), we have 
\[
\bigl| \Omega(Y(y_0,\tau)) \bigr|
\leq 2 |Y(y_0,\tau)-y_*| \leq 4 \delta_*,
\]
for all $\tau \leq \tau_0$.
By combining the three displays above, we deduce that for all $\tau \leq \tau_0$, we have
\begin{align*}
|\Omega(y_0)|
&\leq e^{(\gamma+1)\tau} \cdot 18 e^{-3\gamma \tau}\bigl|(\nabla_a Y) (y_0,\tau_0)\bigr|^2 e^{4 \delta_*(\tau_0-\tau)} \cdot 4 \delta_*
\\
&= \Bigl( 72 \delta_* \bigl|(\nabla_a Y) (y_0,\tau_0)\bigr|^2 e^{4 \delta_*\tau_0} \Bigr) e^{(1-2\gamma-4\delta_*)\tau}.
\end{align*}
By the definition of $\delta_*$ in~\eqref{eq:delta:star:def}, we have that $1-2\gamma-4\delta_* \geq 8 \delta_* > 0$. Thus, upon passing $\tau \to -\infty$ in the above expression, we deduce that $\Omega(y_0) =0$. 

Since $y_0 \in \mathbb{R}^3$ was arbitrary, we deduce $\Omega \equiv 0$. The velocity profile $U$ is thus irrotational, incompressible, and decays at infinity (due to~\eqref{eq:ass:medium:rare}), a contradiction to the assumption that $U$ is nontrivial. Thus, $\gamma \geq 1/2$ must hold.
\end{proof}

\section{Axisymmetric global self-similarity}
\label{sec:axi:sym}
{For smooth initial conditions (for instance $C^1$ vorticity), whether the 3D Euler equations~\eqref{eq:Euler} exhibit axisymmetric singularities remains an open problem. As such, it is natural to revisit the discussion of hypothetical globally self-similar solutions from Section~\ref{sec:self-similar:Euler}, and make the additional assumption} that the  self-similar profile $U(y)$ is axisymmetric around the $\vec{e}_3$ axis. 

{The focus on axisymmetric profiles is also motivated by ongoing computational studies (see e.g.~\cite{PumirSiggia92,GuoLuo14,HouDeHuang22,Hou23,WLGZB23} and references therein), in which solutions $U$ of the full (stationary) self-similar Euler equation~\eqref{eq:profile:U} are sought  within this symmetry class, because this dramatically reduces the computational space domain: from $y \in \mathbb{R}^3$ to $(r,z) \in \mathbb{R}_+ \times \mathbb{R}$.} 

We adopt self-similar cylindrical coordinates $(r, z, \theta)$, where $r = \sqrt{y_1^2 + y_2^2}$ and $z = y_3$. The self-similar velocity profile $U$ can be written in this frame as
\[
U(y) = U_r(r, z) \vec{e}_r + U_\theta(r, z) \vec{e}_\theta + U_z(r, z) \vec{e}_z.
\]
The pressure is $P = P(r,z)$. With this notation, \eqref{eq:profile:U} becomes
\begin{equation}
\label{eq:profile:U:axisym}
\begin{aligned}
(1-\gamma) U_r + (\gamma r + U_r) \partial_r U_r + (\gamma z + U_z) \partial_z U_r  + \partial_r P 
&= \frac{1}{r} U_\theta^2, \\
(1-\gamma) U_z + (\gamma r + U_r) \partial_r U_z + (\gamma z + U_z) \partial_z U_z + \partial_z P
&=0 , \\
(1-\gamma) U_\theta + (\gamma r + U_r) \partial_r U_\theta + (\gamma z + U_z) \partial_z U_\theta
&=-\frac{1}{r} U_r U_\theta ,
\\
\partial_r U_r + \frac{1}{r} U_r + \partial_z U_z &=0.
\end{aligned}
\end{equation}
{
\begin{remark}
\label{rem:yes:swirl}
Because we consider the self-similar collapse of initially smooth solutions, the swirl component of velocity cannot vanish identically: $U_\theta \not \equiv 0$. 
\end{remark}
}

The vorticity vector $\Omega = \nabla \times U$ may be written in components as 
\[
\Omega(y) = \Omega_r(r, z) \vec{e}_r + \Omega_\theta(r, z) \vec{e}_\theta + \Omega_z(r, z) \vec{e}_z,
\]
where
\begin{align}
\Omega_r = - \partial_z U_\theta, 
\qquad
\Omega_z = \partial_r U_\theta + \tfrac{1}{r} U_\theta,
\qquad
\Omega_\theta = \partial_z U_r - \partial_r U_z. 
\label{eq:Omega:components:def}
\end{align}
With this notation, \eqref{eq:profile:Omega} becomes 
\begin{equation}
\label{eq:profile:Omega:axisym}
\begin{aligned}
\Omega_r + (\gamma r + U_r) \partial_r \Omega_r + (\gamma z + U_z) \partial_z \Omega_r    
&= \Omega_r \partial_r U_r + \Omega_z \partial_z U_r, \\
\Omega_z + (\gamma r + U_r) \partial_r \Omega_z + (\gamma z + U_z) \partial_z \Omega_z  
&= \Omega_r \partial_r U_z + \Omega_z \partial_z U_z  , \\
\Omega_\theta + (\gamma r + U_r) \partial_r \Omega_\theta + (\gamma z + U_z) \partial_z \Omega_\theta
&= \frac{1}{r} \bigl(U_r \Omega_\theta - 2 U_\theta \Omega_r \bigr).
\end{aligned}
\end{equation}

\begin{remark}
\label{rem:vanishing:on:axis}
Recall that if $\Omega(y)$ is assumed to be continuous, then on the axis of symmetry $\{r=0\}$ we have $\Omega_r (0,z) = \Omega_\theta(0,z) = 0$. The same reasoning applies for a continuous velocity: $U_r(0,z) = U_\theta(0,z) = 0$.  This leaves open the possibility that $\Omega_z(0,z)$ or   $U_z(0,z)$ do not vanish identically on the axis of symmetry.  
\end{remark}

\subsection{Lagrangian trajectories in the meridional  plane}
The ODE system~\eqref{eq:Lagrangian:SS} for the three-dimensional flow map $Y(a,\tau)$,  decomposes into three components:
\begin{equation}
\label{eq:r:z:dot}
\begin{aligned}
\frac{d}{d\tau} R(r,z,\tau) &= \gamma R(r,z,\tau)  + U_r(R(r,z,\tau),Z(r,z,\tau) ),  \\
\frac{d}{d\tau} Z(r,z,\tau) &= \gamma Z(r,z,\tau) + U_z(R(r,z,\tau), Z(r,z,\tau)), 
\end{aligned}
\end{equation}
and 
\begin{equation}
 \frac{d}{d\tau} \Theta(r,z,\tau) = \frac{1}{R(r,z,\tau)}  U_\theta(R(r,z,\tau), Z(r,z,\tau)),  
 \label{eq:theta:dot}
\end{equation}
with initial conditions $R(r,z,0) = r , Z(r,z,0) = z$, and $\Theta(r,z,0) = 0$. The dynamics of $R$ and $Z$ are decoupled from that of $\Theta$, reducing the 3D Lagrangian evolution to a 2D autonomous system \eqref{eq:r:z:dot}  in the meridional plane $(R, Z)$. The full 3D trajectory is this 2D path spun around the $\vec{e}_3$-axis according to~\eqref{eq:theta:dot}. 

\subsection{Fixed points of the meridional flow with nonzero swirl}
\label{sec:axi:Kelvin}

Here we explore constraints imposed by the existence of a fixed point $(r,z) \in \mathbb{R}_+ \times \mathbb{R}$ for the dynamics~\eqref{eq:r:z:dot} in the meridional plane; i.e., points such that 
\begin{equation}
\label{eq:axisym:fixed}
\gamma r  + U_r (r ,z ) = 0,
\qquad \mbox{and} \qquad 
\gamma z  + U_z(r ,z ) = 0.
\end{equation}
In this case, we automatically deduce $R(r ,z ,\tau) = r $, $Z(r ,z ,\tau) = z $, for all $\tau \in \mathbb{R}$, and $\Theta (r ,z ,\tau) = \frac{\tau}{r } U_\theta(r ,z ) \mod 2\pi$. If we additionally know that $U_\theta(r ,z ) \neq 0$,\footnote{The condition $U_\theta(r ,z ) \neq 0$ implies $r>0$, due to Remark~\ref{rem:vanishing:on:axis}.} we have thus obtained a periodic orbit for the full dynamics of the Lagrangian flow map $Y(\cdot,\tau)$ with nonzero circulation. This leads to the following result:
\begin{theorem}
\label{thm:axissym:with:swirl}
Assume that $U$ is a $C^2$ smooth, axisymmetric, globally self-similar velocity profile for 3D incompressible Euler equations. Assume there exists a fixed point $(r,z)$ for the dynamics in the meridional plane, i.e.~\eqref{eq:axisym:fixed} holds.  If $U_\theta(r,z) \neq 0$, then the similarity exponent satisfies $\gamma = \tfrac 12$.
\end{theorem}

\begin{remark}
If $(r,z)$ is a solution of~\eqref{eq:axisym:fixed} with $U_\theta(r,z) =0$, then also $\Omega_\theta(r,z)=0$. This follows from the swirl component of the vorticity equation~\eqref{eq:profile:Omega:axisym}.
\end{remark}

\begin{proof}[Proof of Theorem~\ref{thm:axissym:with:swirl}]
We consider the loop $C(0)$ parameterized by $ (r,z, \theta )$, with $\theta \in [0,2\pi)$. Due to~\eqref{eq:axisym:fixed}, this loop is an invariant set of the flow: $C(\tau) = C(0)$ because the Lagrangian flow~\eqref{eq:r:z:dot}--\eqref{eq:theta:dot} merely advects points along the loop.  Combining~\eqref{eq:Kelvin:SS} and~\eqref{eq:axisym:fixed},  after a short computation we deduce that 
\[
\Gamma_{\rm ss}(0) = e^{(1-2\gamma)\tau}\Gamma_{\rm ss}(\tau) = e^{(1-2\gamma ) \tau} \int_{\mathbb{T}} U_\theta(r,z) r d\lambda = e^{(1-2\gamma ) \tau} \cdot 2\pi r U_\theta(r,z).
\]
Since we have assumed $ U_\theta(r ,z )\neq 0$ we must have that $r>0$, and we deduce that $\gamma = 1/2$ is the only exponent which conserves the circulation along this loop.
\end{proof}

\subsection{No assumptions on the meridional flow}
The main result of this section shows that  solutions of~\eqref{eq:profile:U:axisym} which have $C^2$ smooth velocity profiles automatically must satisfy $\gamma \geq 1/2$, with \emph{no further assumptions} on the meridional flow.

\begin{theorem}
\label{thm:axisym:final}
Let $U$ be a nontrivial $C^2$ smooth axisymmetric globally self-similar velocity profile for the 3D incompressible Euler equations~\eqref{eq:profile:U:axisym}, satisfying~\eqref{eq:U:non:negotiable} and~\eqref{eq:ass:medium:rare}.  Then $\gamma \geq 1/2$. 
\end{theorem}
\begin{proof}[Proof of Theorem~\ref{thm:axisym:final}]
Throughout this proof we work with trajectories of the meridional flow; that is, we analyze the solution $(R(\tau),Z(\tau)) = (R(r_0,z_0,\tau),Z(r_0,z_0,\tau))_{\tau \in \mathbb{R}}$ of the ODE~\eqref{eq:r:z:dot}, with initial condition $(r_0,z_0)$ at $\tau=0$. 

Since $U_r(0,z)=0$ for all $z \in \mathbb{R}$ (see~Remark~\ref{rem:vanishing:on:axis}) and the ODE~\eqref{eq:r:z:dot} is driven by a locally Lipschitz autonomous vector field, the axis of symmetry $\{r=0\}$ is invariant under the meridional flow. Moreover, a trajectory originating at $(r_0,z_0)$ with $r_0>0$ cannot touch the axis $\{r=0\}$ in finite time; if it did so, it would have to coincide with a trajectory confined to the axis for all time.

Recall from~\eqref{eq:R:flat:def} that if $r^2+ z^2 \geq R_\flat^2$, then   $|r U_r(r,z) + z U_z(r,z)| \leq \frac{\gamma}{2} (r^2 + z^2)$. This bound implies that for any fixed $(r_0,z_0) \in \mathbb{R}_+\times \mathbb{R}$, the trajectory $(R(\tau),Z(\tau)) = (R(r_0,z_0,\tau),Z(r_0,z_0,\tau))$ satisfies 
\begin{equation}
\label{eq:backwards:in:time:remains:in:compact}
R(\tau)^2 + Z(\tau)^2 \leq \max\{ r_0^2+ z_0^2 , R_\flat^2\}, 
\qquad
\mbox{for all}
\qquad
\tau \leq 0. 
\end{equation}
In order to prove~\eqref{eq:backwards:in:time:remains:in:compact}, note that by~\eqref{eq:r:z:dot} we have that the function $\rho(\tau)= \sqrt{R(\tau)^2 + Z(\tau)^2}$, satisfies $\frac 12 \frac{d}{d\tau} \rho^2 = \gamma \rho^2 + (r U_r + z U_z)(R(\tau),Z(\tau))$. Thus, if $\rho  \geq  R_\flat$, then $ \frac{d}{d\tau} \rho^2 \geq \gamma \rho^2 \geq \gamma R_\flat^2>0$, which proves~\eqref{eq:backwards:in:time:remains:in:compact} via a standard continuity argument.

Combining the two statements above, we have thus established that for any $r_0>0$ and $z_0 \in \mathbb{R}$, the backwards-in-time trajectory $(R(\tau),Z(\tau))=(R(r_0,z_0,\tau),Z(r_0,z_0,\tau))_{\tau \leq 0}$ is confined  to the compact set $K_0 := \{ (r,z) \in \mathbb{R}_+ \times \mathbb{R} \colon r^2 + z^2 \leq \max\{ r_0^2 + z_0^2, R_\flat^2\} \}$, and $R(r_0,z_0,\tau)>0$ for all $\tau \leq 0$.

For the remainder of the proof, assume \emph{by contradiction} that $\gamma < 1/2$. 

We note that the $U_\theta$ equation in~\eqref{eq:profile:U:axisym} gives 
\[
(\gamma r + U_r) \partial_r (r U_\theta) 
+ 
(\gamma z + U_z) \partial_z (r U_\theta)
+
(1-2\gamma) (r U_\theta)
=
0.
\]
For a fixed $(r_0,z_0)\in\mathbb{R}_+\times \mathbb{R}$ with $r_0>0$, we integrate this equation along the backwards trajectory $( R(\tau),Z(\tau) )_{\tau \leq 0}$ to obtain
\[
e^{(1-2\gamma) \tau} R(\tau) U_\theta(R(\tau),Z(\tau)) 
=
r_0 U_\theta(r_0,z_0),
\qquad
\mbox{for all}
\qquad
\tau\leq 0.
\]
Since $\gamma<1/2$, we have $e^{(1-2\gamma) \tau}  \to 0$ as $\tau\to - \infty$. Moreover, we have previously shown that  $(R(\tau),Z(\tau)) \in K_0$ for all $\tau\leq 0$, and the continuous function $rU_\theta(r,z)$ is bounded on the compact set $K_0$. Therefore, upon passing $\tau \to -\infty$ in the above display, and using that $r_0>0$, we deduce that $U_\theta(r_0,z_0)=0$; this fact is a manifestation of the conservation of circulation~\eqref{circinv}. 
With Remark~\ref{rem:vanishing:on:axis}, it follows that $U_\theta \equiv 0$ on $\mathbb{R}_+\times\mathbb{R}$. Since $U$ is $C^1$ smooth, from~\eqref{eq:Omega:components:def} we also obtain $\Omega_r \equiv \Omega_z \equiv 0$ on $\mathbb{R}_+\times\mathbb{R}$.

To conclude, we use the $\Omega_\theta$ equation in~\eqref{eq:profile:Omega:axisym} and the fact that $U_\theta\equiv0$ to deduce
\[
(\gamma r + U_r) \partial_r \bigl(\tfrac{1}{r} \Omega_\theta\bigr)
+ 
(\gamma z + U_z) \partial_z \bigl(\tfrac{1}{r} \Omega_\theta\bigr)
+
(1+\gamma) \bigl(\tfrac{1}{r} \Omega_\theta\bigr)
=
0.
\]
For a fixed $(r_0,z_0)\in\mathbb{R}_+\times \mathbb{R}$ with $r_0>0$, we integrate this equation along the backwards trajectory $( R(\tau),Z(\tau) )_{\tau \leq 0}$ to obtain
\[
e^{(1+\gamma) \tau} \tfrac{1}{R(\tau)} \Omega_\theta(R(\tau),Z(\tau)) 
=
\tfrac{1}{r_0} \Omega_\theta(r_0,z_0).
\]
Since $\gamma>0$, we have $e^{(1+\gamma) \tau}  \to 0$ as $\tau\to - \infty$.  Moreover, we have previously shown that  $(R(\tau),Z(\tau)) \in K_0$ for all $\tau \leq 0$, and  the function $\frac{1}{r} \Omega_\theta(r,z)$ is bounded on the compact set $K_0$.\footnote{Since $\Omega$ is $C^1$ smooth and $\Omega_\theta(0,z) = 0$, we have that $\frac{1}{r} \Omega_\theta(r,z) = \frac{1}{r} \int_0^r (\partial_r \Omega_\theta)(r^\prime,z) dr^\prime$. For any $(r,z) \in K_0$, the whole line segment $(r^\prime,z)_{0\leq r^\prime\leq r}$ lies in $K_0$, and therefore $\frac{1}{r} |\Omega_\theta(r,z)| \leq \sup_{K_0}|\partial_r \Omega_\theta| < \infty$.}
Upon passing $\tau\to -\infty$ in the above display and using that $r_0>0$, we  deduce $\Omega_\theta(r_0,z_0)=0$. 
With Remark~\ref{rem:vanishing:on:axis}, we also know that $\Omega_\theta$ vanishes on the axis of symmetry, and hence  $\Omega_\theta \equiv 0$ on $\mathbb{R}_+ \times \mathbb{R}$. 

With $\Omega_r\equiv\Omega_z\equiv\Omega_\theta \equiv 0$  on $\mathbb{R}_+ \times \mathbb{R}$, by axisymmetry and continuity we deduce that $U$ is irrotational and incompressible on $\mathbb{R}^3$. Note however that~\eqref{eq:ass:medium:rare} implies $|\nabla U(y)| \to 0$ as $|y|\to \infty$ for any $\gamma>0$, and thus $U$ must be a constant. Our normalization $U(0)=0$ (cf.~\eqref{eq:U:non:negotiable}) then proves $U$ is trivial, which contradicts the assumption that $U$ is nontrivial. Hence, the contradiction ansatz fails, and we must have $\gamma \geq 1/2$.
\end{proof}

\begin{remark}
Inspecting the proof of Theorem~\ref{thm:axisym:final}, we note that the contradiction ansatz $\gamma < 1/2$ was used \emph{only once}: we appealed to $\gamma < 1/2$ in the form $e^{(1-2\gamma)\tau} \to 0$ as $\tau\to -\infty$, in order to conclude $U_\theta \equiv 0$. The remainder of the proof only uses $\gamma>0$. Thus, we have proven that there is \emph{no swirl-free},  $C^2$ smooth, axisymmetric, globally self-similar velocity profile $U$ for the 3D incompressible Euler equations~\eqref{eq:profile:U:axisym}, satisfying~\eqref{eq:U:non:negotiable} and~\eqref{eq:ass:medium:rare} for \emph{any} $\gamma>0$. 
\end{remark}

\section*{Acknowledgements}  
The work of P.C. was partially supported by NSF grant DMS-2106528 and by a Simons Collaboration
Grant 601960. The work of M.I. was partially supported by NSF grant DMS-2204614. The work of V.V. was partially supported by the Collaborative NSF grant DMS-2307681
and a Simons Investigator Award.


\begin{thebibliography}{99}


\bibitem{Barenblatt96} G.I.~Barenblatt.
\newblock Scaling, self-similarity, and intermediate asymptotics: dimensional analysis and intermediate asymptotics. 
\newblock No. 14. Cambridge University Press, 1996.

\bibitem{BKM84}
J.T.~Beale, T.~Kato, A.~Majda. 
\newblock Remarks on the breakdown of smooth solutions for the 3-D Euler equations. 
\newblock \textit{Communications in Mathematical Physics}~\textbf{94} (1):61--66, 1984.

\bibitem{BronziShvydkoy15}
A.~Bronzi, R.~Shvydkoy. 
\newblock On the energy behavior of locally self-similar blowup for the Euler equation.
\newblock \textit{Indiana University Mathematics Journal}~\textbf{64} (5):1291--1302, 2015.

\bibitem{BCG25}
T.~Buckmaster, G.~Cao-Labora,  J.~G{\'o}mez-Serrano.
\newblock Smooth imploding solutions for 3d compressible fluids.
\newblock \textit{Forum of Mathematics, Pi}, 13:e6, 2025.

\bibitem{BSV22}
T.~Buckmaster, S.~Shkoller, V.~Vicol.
\newblock Formation of shocks for 2D isentropic compressible Euler. 
\newblock \textit{Comm. Pure Appl. Math.}~\textbf{75} (9):2069--2120, 2022

\bibitem{BSV23}
T.~Buckmaster, S.~Shkoller, V.~Vicol.
\newblock Formation of point shocks for {3D} compressible {E}uler. 
\newblock \textit{Comm. Pure Appl. Math.}~\textbf{76} (9):2073--2191, 2023.

\bibitem{Chae07}
D.~Chae. 
\newblock Nonexistence of self-similar singularities for the 3D incompressible Euler equations.
\newblock \textit{Communications in Mathematical Physics}~\textbf{273} (1):203--215, 2007.


\bibitem{ChaeWolf17}
D.~Chae, J.~Wolf.
\newblock Removing discretely self-similar singularities for the 3D Navier-Stokes equations.
\newblock \textit{Comm. Partial Differential Equations}~\textbf{42} (9):1359--1374, 2017.

\bibitem{Chen24}
J.~Chen. 
\newblock{Remarks on the smoothness of the $C^{1,\alpha}$ asymptotically self-similar singularity in the 3D Euler and 2D Boussinesq equations}. 
\newblock \textit{Nonlinearity}~\textbf{37}, no. 6:065018, 2024.

\bibitem{CCSV24}
J.~Chen, G.~Cialdea, S.~Shkoller, V.~Vicol. 
\newblock Vorticity blowup in 2D compressible Euler equations.
\newblock \textit{Duke Math. J.}, to appear. \textit{arXiv preprint} \href{https://arxiv.org/abs/2407.06455}{arXiv:2407.06455}, 2024.

\bibitem{ChenHou21}
J.~Chen, T.Y.~Hou. 
\newblock Finite time blowup of 2D Boussinesq and 3D Euler equations with $C^{1,\alpha}$ velocity and boundary. 
\newblock \textit{Communications in Mathematical Physics}~\textbf{383} (3):1559--1667, 2021.

\bibitem{ChenHou22}
J.~Chen, T.Y.~Hou. 
\newblock Stable nearly self-similar blowup of the 2D Boussinesq and 3D Euler equations with smooth data I: Analysis.
\newblock \textit{arXiv preprint} \href{https://arxiv.org/abs/2210.07191}{arXiv:2210.07191}, 2022.

\bibitem{ChenHou25a}
J.~Chen, T.Y.~Hou. 
\newblock Stable nearly self-similar blowup of the 2D Boussinesq and 3D Euler equations with smooth data II: Rigorous numerics.
\newblock \textit{Multiscale Modeling \& Simulation}~\textbf{23} (1):25--130,  2025.

\bibitem{ChenHou25}
J.~Chen, T.Y.~Hou. 
\newblock Singularity formation in 3D Euler equations with smooth initial data and boundary.
\newblock \textit{Proc. Natl. Acad. Sci. U.S.A.}~\textbf{122} (27) e2500940122, 2025.

\bibitem{child} 
S.~Childress, G.R.~Ierley, E.A.~Spiegel, W.R.~Young. 
\newblock Blow-up of unsteady two-dimensional Euler and Navier-Stokes solutions having stagnation-point form. 
\newblock \textit{Journal of Fluid Mechanics}~\textbf{203}:1--22, 1989.

\bibitem{CMZ25}
D.~Cordoba, L.~Martinez-Zoroa, F.~Zheng.
\newblock Finite Time Singularities to the 3D Incompressible Euler Equations for Solutions {in $C^\infty(\mathbb{R}^3\setminus\{0\})\cap C^{1,\alpha}\cap L^2$}.
\newblock \textit{Annals of PDE}~\textbf{11} (2), p.19, 2025.

\bibitem{ConstantinArea}
P.~Constantin. 
\newblock Navier-Stokes equations and area of interfaces.
\newblock \textit{Communications in Mathematical Physics}~\textbf{129} (2):241--266, 1990.

\bibitem{Constantin94a} P.~Constantin.
\newblock Geometric and Analytic Studies in Turbulence. 
\newblock In: Sirovich, L. (eds) \textit{Trends and Perspectives in Applied Mathematics}. Applied Mathematical Sciences, vol.~100. Springer, New York, 1994. 

\bibitem{Constantin17} P.~Constantin. 
\newblock Analysis of Hydrodynamic Models.
\newblock CBMS-NSF Regional Conference Series in Applied Mathematics~\textbf{90}. SIAM, Philadelphia, 2017.

\bibitem{Cric}
P.~Constantin.
\newblock The Euler equations and nonlocal conservative Riccati equations. 
\newblock \textit{International Mathematics Research Notices}~\textbf{9}:455--465, 2000.

\bibitem{ConstantinFefferman93} 
P.~Constantin, C.~Fefferman. 
\newblock Direction of vorticity and the problem of global regularity for the Navier-Stokes equations. 
\newblock \textit{Indiana University Mathematics Journal}~\textbf{42} (3):775--789, 1993.

\bibitem{CFbook} 
P.~Constantin, C.~Foias.
\newblock {Navier-Stokes Equations}.
\newblock University of Chicago Press, Chicago, 1988.
 

\bibitem{EggersFontelos15} 
J.~Eggers, M.A.~Fontelos.
\newblock Singularities: formation, structure, and propagation.
\newblock Cambridge Texts in Applied Mathematics. Cambridge University Press, Cambridge, 2015.  

\bibitem{Elgindi21}
T.M.~Elgindi.
\newblock Finite-time Singularity Formation for $C^{1,\alpha}$ Solutions to the Incompressible Euler Equations on $\mathbb{R}^3$. 
\newblock \textit{Annals of Mathematics}~\textbf{194} (3):647--727,  2021.

\bibitem{Elgindi25}
T.M.~Elgindi. 
\newblock Dynamics of Ideal Fluid Flows.
\newblock In \emph{International Congress of Mathematicians 2026},  Vol.~5, 285--306. Society for Industrial and Applied Mathematics, 2026.


\bibitem{EGM22}
T.M.~Elgindi, T.-E.~Ghoul, N.~Masmoudi. 
\newblock On the stability of self-similar blow-up for $C^{1,\alpha}$ solutions to the incompressible Euler equations on $\mathbb{R}^3$. 
\newblock \textit{Cambridge Journal of Mathematics}~\textbf{9} (4):1035--1075, 2022.

\bibitem{ElgindiPasqua23}
T.M.~Elgindi, F.~Pasqualotto. 
\newblock From instability to singularity formation in incompressible fluids.
\newblock \textit{arXiv preprint} \href{https://arxiv.org/abs/2310.19780}{arXiv:2310.19780}, 2023.

\bibitem{gib}
J.D.~Gibbon, K.~Ohkitani. 
\newblock Numerical study of singularity formation in a class of Euler and Navier-Stokes flows.
\newblock \textit{Phys. Fluids}~\textbf{8}:1744--1752, 1996.


\bibitem{Guderley42}
K.G.~Guderley.
\newblock Starke kugelige und zylindrische Verdichtungsst{\"o}sse in der
  N{\"a}he des Kugelmittelpunktes bzw. der Zylinderachse.
\newblock \textit{Luftfahrtforschung}~\textbf{19}:302--312, 1942.

\bibitem{HouDeHuang22}
T.Y.~Hou, D.~Huang.
\newblock A potential two-scale traveling wave singularity for 3D incompressible Euler equations.
\newblock \textit{Physica D: Nonlinear Phenomena}~\textbf{435}:133257, 2022.

\bibitem{Hou23}
T.Y.~Hou.
\newblock Potential Singularity of the 3D Euler Equations in the Interior Domain.
\newblock \textit{Foundations of Computational Mathematics}~\textbf{23}:2203--2249, 2023. 


\bibitem{Hunter60}
C.~Hunter.
\newblock On the collapse of an empty cavity in water.
\newblock \textit{J. Fluid Mech.}, 8(2):241--263, 1960.

\bibitem{GuoLuo14}
G.~Luo, T.Y.~Hou. 
\newblock Potentially singular solutions of the 3D axisymmetric Euler equations.
\newblock \textit{Proc. Natl. Acad. Sci. U.S.A.}~\textbf{111}:12968--12973, 2014.

\bibitem{MRRS22a}
F.~Merle, P.~Rapha{\"e}l, I.~Rodnianski,   J.~Szeftel.
\newblock On the implosion of a compressible fluid {I}: {S}mooth self-similar inviscid profiles.
\newblock \textit{Ann. of Math. (2)}, 196 (2):567--778, 2022.

\bibitem{MRRS22b}
F.~Merle, P.~Rapha{\"e}l, I.~Rodnianski,   J.~Szeftel.
\newblock On the implosion of a compressible fluid II: Singularity formation.
\newblock \textit{Ann. of Math. (2)}, 196 (2):779--889, 2022.

\bibitem{NRS96}
J.~Ne\v{c}as, M.~R\r{u}\v{z}i\v{c}ka, V.~\v{S}ver\'ak. 
\newblock On Leray's self-similar solutions of the Navier-Stokes equations.
\newblock \textit{Acta Math.}~{\bf 176}:283--294, 1996.


\bibitem{PumirSiggia92}
A.~Pumir, E.D.~Siggia. 
\newblock Development of singular solutions to the axisymmetric Euler equations. 
\newblock \textit{Phys. Fluids A}~\textbf{4} (7):1472--1491, 1992.


\bibitem{Sedov46}
L.I.~Sedov.
\newblock Propagation of strong shock waves.
\newblock \textit{Prikladnaya Matematika i Mekhanika}, 10 (2):241--250, 1946.

\bibitem{Sedov18}
L.I.~Sedov.
\newblock Similarity and dimensional methods in mechanics.
\newblock CRC press, 2018.

\bibitem{Seregin05}
G.~Seregin.
\newblock On smoothness of $L_{3,\infty}$-solutions to the Navier-Stokes equations up to boundary.
\newblock \textit{Math. Ann.}~\textbf{332} (1):219--238, 2005. 

\bibitem{SWWZ25}
F.~Shao, D.~Wei, S.~Wang, Z.~Zhang. 
\newblock Blow-up of the 3-D compressible {N}avier-{S}tokes equations for monatomic gases.
\newblock \textit{arXiv preprint} \href{https://arxiv.org/abs/2501.15701}{arXiv:2501.15701}, 2025.

\bibitem{Shvydkoy13}
R.~Shvydkoy. 
\newblock A study of energy concentration and drain in incompressible fluids.
\newblock \textit{Nonlinearity}~\textbf{26} (2):425--436, 2013.

\bibitem{jtstuart} 
J.T.~Stuart. 
\newblock Singularities in three-dimensional compressible Euler flows with vorticity.
\newblock \textit{Proc. Roy. Soc. London Ser. A} 417:91--108, 1988.

\bibitem{Taylor50}
G.I.~Taylor.
\newblock The formation of a blast wave by a very intense explosion. {I}. {T}heoretical discussion.
\newblock \textit{Proceedings of the Royal Society of London. Series A.
  Mathematical and Physical Sciences}, 201 (1065):159--174, 1950.
  
\bibitem{Tsai98}
T.-P.~Tsai.
\newblock On Leray's self‐similar solutions of the Navier‐Stokes equations satisfying local energy estimates.
\newblock \textit{Arch. Rational Mech. Anal.}~\textbf{143} (1):29--51, 1998.  

\bibitem{vonNeumann47}
J.~von Neumann.
\newblock The point source solution.
\newblock In R.~S. Brode, J.~O. Hirschfelder, and S.~R. Brinkley (eds). 
\textit{Shock Hydrodynamics and Blast Waves}, pages 27--55. Los Alamos 
  Scientific Laboratory Report LA-2000, 1947.
  
\bibitem{WLGZB23}  
Y.~Wang, C.Y.~Lai, J.~Gomez-Serrano, T.~Buckmaster. 
\newblock Asymptotic self-similar blow-up profile for three-dimensional axisymmetric Euler equations using neural networks. 
\newblock \textit{Phys. Rev. Lett.}~\textbf{130}, 244002, 2023.

\bibitem{WangEtAl25}
Y.~Wang, et al.
\newblock Discovery of unstable singularities.
\newblock \textit{arXiv preprint} \href{https://arxiv.org/abs/2509.14185}{arXiv:2509.14185}, 2025.

\bibitem{WLLB25}
Y.~Wang, T.~Leger, C.Y.~Lai, T.~Buckmaster. 
\newblock Resolving Sharp Gradients of Unstable Singularities to Machine Precision via Neural Networks. 
\newblock \textit{arXiv preprint} \href{https://arxiv.org/abs/2511.22819}{arXiv:2511.22819}, 2025.

\bibitem{WLLAH25}
Y.~Wang, Z.~Liu, Z.~Li, A.~Anandkumar, T.Y.~Hou. 
\newblock High precision PINNs in unbounded domains:
application to singularity formulation in PDEs. 
\newblock \textit{arXiv preprint} \href{https://arxiv.org/abs/2506.19243}{arXiv:2506.19243}, 2025.


\end{thebibliography}
\end{document}